\documentclass{amsart}
\usepackage{amssymb}
\usepackage[mathscr]{eucal}
\usepackage{eufrak}
\usepackage{xspace}
\newcommand{\R}{\mathbb{R}}

\newcommand{\Z}{\mathbb{Z}}
\newcommand{\N}{\mathbb{N}}
\newcommand{\C}{\mathbb{C}}
\newcommand{\LL}{\mathcal{L}}

\newcommand{\ve}{\varepsilon}

\newcommand{\loc}{{\text{\rm loc}}}
\newcommand{\X}{\times}

\newcommand{\del}{{\partial}}

\renewcommand{\d}{\delta}
\renewcommand{\l}{\lambda}

\renewcommand{\a}{\alpha}
\renewcommand{\b}{\beta}

\newcommand{\s}{\sigma}
\newcommand{\g}{\gamma} 
\newcommand{\z}{\zeta}
\renewcommand{\k}{\kappa}

\newcommand{\Om}{\Omega}
\newcommand{\om}{\omega}

\newcommand{\supp}{\text{\rm supp}\,}

\renewcommand{\div}{\text{\rm div}}

\renewcommand{\supp}{\text{\rm supp}\,}

\newcommand{\m}{\mathbf{m}}

\newcommand{\cE}{{\mathcal E}}
\newcommand{\cL}{{\mathcal L}}

\newcommand{\bbf}{\mathbf{b}}
\newcommand{\abf}{\mathbf{a}}

\newcommand{\Abf}{\mathbf{A}}
\newcommand{\Bbf}{\mathbf{B}}

\newcommand{\Sbf}{\mathbf{S}}

\newcommand{\Sbb}{\mathbb S}

\newcommand{\bbE}{{\mathbb E}}
\newcommand{\bbT}{{\mathbb T}}
\newcommand{\bbP}{{\mathbb P}}

\newcommand{\dist}{\text{dist\,}}

\theoremstyle{plain}
\newtheorem{theorem}{Theorem}[section]
\newtheorem{corollary}{Corollary}[section]

\newtheorem{lemma}{Lemma}[section]
\newtheorem{proposition}{Proposition}[section]
\theoremstyle{definition}
\newtheorem{definition}{Definition}[section]
\theoremstyle{remark}

\newtheorem{remark}{Remark}[section]
\numberwithin{equation}{section}
\usepackage{mathtools}

\newcommand{\wt}{\widetilde}
\newcommand{\ff}{\mathfrak{f}}
\newcommand{\bb}{\mathfrak{b}}

\newcommand{\qq}{\mathfrak{q}}
\newcommand{\gf}{\mathfrak{g}}
\newcommand{\FF}{\mathfrak{F}}
\newcommand{\vf}{\mathfrak{v}}

\newcommand{\rr}{\mathfrak{r}}
\newcommand{\sg}{\mathfrak{s}}
\newcommand{\cc}{\mathfrak{c}}

\newcommand{\hh}{\mathfrak{h}}

\newcommand{\meas}{{\text{\rm meas}}}
\newcommand{\nn}{\mathfrak{n}}
\newcommand{\dd}{\mathfrak{d}}

\newcommand{\zz}{\mathfrak{z}}
\usepackage{mathrsfs}
\newcommand{\Cc}{\mathscr{C}}
\newcommand{\Dd}{\mathscr{D}}
\newcommand{\Ss}{\mathscr{S}}
\newcommand{\Ff}{\mathscr{F}}
\newcommand{\lf}{\mathfrak{l}}
\newcommand{\af}{\mathfrak{a}}
\newcommand{\pp}{\mathfrak{p}}

\newcommand{\JJ}{\mathscr{J}}

\author[J. F. Nariyoshi]{Jo\~ao Fernando ~Nariyoshi}
\address{Instituto de Matem\'atica Pura e Aplicada - IMPA\\ Estrada Dona Castorina, 110\\
	Rio de Janeiro, RJ, 22460-320, Brazil}
\email{jfcn@impa.br}

\title{Critical velocity averaging lemmas}

\begin{document}

\begin{abstract}
	We prove new velocity averaging lemmas for multi-dimensional hyperbolic-parabolic partial differential equations. These theorems may be applied to establish several compactness results for both deterministic and stochastic convection-diffusion equations. Among the strengths of our theory is the criticality of the source term, which may include spatial derivatives of second order and stochastic noises. 
\end{abstract}

\maketitle

\section{Introduction}

We are concerned with the relative compactness of the so-called \textit{velocity averages}
\begin{equation}
\int_\R f(t,x,v)\eta(v)\,dv, \label{0.-1}
\end{equation}
where $\eta : \R \to \R$ is a weight function, $(t,x,v) \in \R\X\R^N\X\R$, and $f(t,x,v)$ obeys a second-order multidimensional parabolic-hyperbolic equation of the general form
\begin{equation}
\frac{\del f}{\del t} + \abf(v) \cdot \nabla_x f - \div_x (\bbf(v) \nabla_x f) = \Lambda + \Phi \, \frac{dW}{dt}, \label{0.0}
\end{equation}
in which $\abf : \R \to \R^N$ is a convection vector function, $\bbf : \R \to \mathscr{L}(\R^N)$ is a nonnegative diffusion matrix function, $\Lambda(t,x,v)$ is a tempered distribution, $W(t)$ is a cylindrical Wiener process, and $\Phi(t,x,v)$ are diffusion coefficients.

Theorems in this direction, traditionally named \textit{velocity averaging lemmas}---or just \textit{averaging lemmas}---, are of great interest to the discipline of both deterministic and stochastic degenerate convection-diffusion equations such as
\begin{equation}
\frac{\del \varrho}{\del t} + \sum_{j = 1}^N \frac{\del}{\del x_j} \Abf_j(\varrho) - \sum_{j,k=1}^N \frac{\del^2}{\del x_j \del x_k} \Bbf_{jk}(\varrho) = \mathbf S(\varrho), \label{0.1}
\end{equation}
where $\varrho(t,x) \in \R$ is the unknown field, $\Abf = (\Abf_1, \ldots, \Abf_N): \R \to \R^N$ is a flux function, $\Bbf = (\Bbf_{jk})_{1 \leq j,k \leq N} : \R \to \mathscr{L}(\R^N)$ is such that $\Bbf'(v) \geq 0$ everywhere, and $\Sbf$ represents a forcing source term. Indeed, although there does not exist smooth solutions to Equation \eqref{0.1} and its weak solutions lack uniqueness properties in general, one can introduce the concept of an entropy solution, which, besides allowing one to develop a satisfactory mathematical theory, can be shown to be physically relevant; see, e.g., \textsc{S. N. Kruzkov} \cite{K} and \textsc{J. Carrillo} \cite{C}. This last characteristic is profoundly significant, once this equation and its variants model several important natural phenomena, including sedimentation--consolidation processes (see, e.g., \textsc{M. C. Bustos} \textit{et al.} \cite{BCBT}), the fluid motion through a porous medium (see, e.g., \textsc{G. Chavent--J. Jaffre} \cite{CJ}, and \textsc{J. L. Vazquez} \cite{Vz}), etc. Furthermore, from a theoretical standpoint, it is well-known since the groundbreaking works of \textsc{P.-L. Lions--B. Perthame--E. Tadmor} \cite{LPT} and \textsc{G.-Q. Chen--B. Perthame} \cite{CP} (see also \textsc{A. Debussche--M. Hofmanová--J. Vovelle} \cite{DHV}) that this entropy property in particular also entails that Equation \eqref{0.1} enjoys a kinetic notion of a solution. In this kinetic formulation, one defines a new variable $v \in \R$---commonly known as \textit{velocity}---such that, under a change of variables $\varrho \mapsto f$, $f(t,x,v)$ now solves an equation like \eqref{0.0} with $\abf(v) = \Abf'(v)$ and $\bbf(v) = \Bbf'(v)$. For $\varrho(t,x)$ may be then reconstructed from $f(t,x,v)$ via the integral
$$\varrho(t,x) = \int_\R f(t,x,v) \,dv,$$
one can thus deduce nontrivial regularity properties of the original solutions $\varrho$ by means of velocity averaging lemmas. 

So as to illustrate this point, the argument just delineated allowed several authors to successfully establish
\begin{enumerate}
	\item[(a)] the existence of solutions employing the vanishing viscosity method (see, e.g., \textsc{R. Bürgers--H. Frid--K. H. Karlsen} \cite{BFK}, \textsc{H. Frid}--\textsc{Y. Li} \cite{FL}, \textsc{B. Gess--M. Hofmanová} \cite{GH}),
	\item [(b)] the strong trace property (see, e.g., \textsc{A. Vasseur} \cite{Va}, \textsc{Y.-S. Kwon--A. Vasseur} \cite{KV}, \textsc{H. Frid--Y. Li} \cite{FL}, and \textsc{H. Frid} \textit{et al.} \cite{FL1}), which is crucial to prove the uniqueness of solutions to many initial-boundary value problems (see the discussion in \cite{Va}),
	\item [(c)] the existence of an asymptotic state (see, e.g., \textsc{G.-Q. Chen--H. Frid} \cite{CF0', CF0}, and \textsc{G.-Q. Chen--B. Perthame} \cite{CP1}),
	\item [(d)] the Sobolev regularity of entropy solutions (see, e.g., \textsc{P.-L. Lions--B. Perthame--E. Tadmor} \cite{LPT}, \textsc{T. Tao--E. Tadmor} \cite{TT}, \textsc{B. Gess--M. Hofmanová} \cite{GH}, \textsc{B. Gess--X. Lamy} \cite{GL}, \textsc{B. Gess} \cite{Ge}, and \textsc{B. Gess--J. Sauer--E. Tadmor} \cite{GST}),
\end{enumerate}
among other propositions. Specially, averaging lemmas are valuable when studying boundary value problems, once in this setting one may not be able to prove akin compactness results via the classical arguments of $L^1$--contractivity.

Velocity averaging lemmas possess a rich, albeit relatively short history, beginning with the original works of \textsc{V. I. Agoshkov} \cite{A} and \textsc{C. Bardos} \textit{et al.} \cite{BGPS} on transport equations. Later on, these averaging lemmas for transport equations were further delved into by \textsc{F. Golse} \textit{et al.} \cite{GLPS}, \textsc{R. J. DiPerna--P.-L. Lions} \cite{DL1, DL2} (with applications to the Boltzmann and Vlasov--Maxwell equations), \textsc{R. J. DiPerna--P.-L. Lions--Y. Méyer} \cite{DLM} (with a general, noncritical source term in $L^p$), \textsc{M. Bézard} \cite{B} and \textsc{P.-L. Lions} \cite{Li} (both of the latter studying optimal regularity in Sobolev spaces), \textsc{B. Perthame--P. Souganidis} \cite{PS} (with a general critical source term in $L^p$), \textsc{R. DeVore--G. Petrova} \cite{DP} (establishing optimal regularity in Besov spaces), \textsc{L. Saint-Raymond} \cite{SR} and \textsc{F. Golse--L. Saint-Raymond} \cite{GSR1, GSR2} (in an $L^1$--framework and with important consequences to the Navier--Stokes equations), \textsc{P.-E. Jabin--H.-Y. Lin--E. Tadmor} \cite{JLT} (using commutator techniques), and \textsc{D. Arsénio--N. Lerner} \cite{AL} (employing an energy method).

The first applications to nonlinear conservation laws were given by \textsc{P.-L. Lions--B. Perthame--E. Tadmor} \cite{LPT} with the introduction of the celebrated kinetic formulation. Their results were subsequently extended by the aforementioned work of \textsc{B. Perthame--P. Souganidis} \cite{PS}, \textsc{P.-E. Jabin--B. Perthame} \cite{JP} (see also \textsc{P.-E. Jabin--L. Vega} \cite{JV1, JV2} for a similar theorems), \textsc{M. Westdickenberg} \cite{W}, and \textsc{F. Berthelin--S. Junca} \cite{BJ}, just to name a few.

Let us also point out that an $L^2$--theory of averaging lemmas for general partial differential operators was developed by \textsc{P. Gérard} \cite{G1,G2} and \textsc{P. Gérard--F. Golse} \cite{GG}  using techniques of $H$-measures (see also \textsc{M. Lazar--D. Mitrović} \cite{LM}). Additionally, it is equally worth mentioning the applications of velocity averaging lemmas to numerical schemes by \textsc{L. Desvillettes--S. Mischler} \cite{DM}, \textsc{S. Mischler} \cite{M}, \textsc{F. Bouchut--L. Desvillettes} \cite{BD}, \textsc{T. Horsin--S. Mischler--A. Vasseur} \cite{HMV}, and \textsc{N. Ayi--T. Goudon} \cite{AG}.

The vast majority of the aforesaid works were restricted to first-order equations, with notable exceptions being: some statements in \textsc{P.-L. Lions--B. Perthame--E. Tadmor} \cite{LPT} regarding hyperbolic-parabolic equations, the abstract theory of \textsc{P. Gérard} \cite{G1,G2} with \textsc{F. Golse} \cite{GG}, and the parabolic averaging lemma of \textsc{M. Lazar--D. Mitrović} \cite{LM}. Indeed, the study of velocity averaging lemmas for convection-diffusion equations has contrastingly a much smaller body of literature and is largely influenced by the towering theory of \textsc{E. Tadmor--T. Tao} \cite{TT}. Their results delved into the Sobolev regularity of entropy solutions to such second-order equations, and they were based on dyadic partitions of the frequency space in terms of the Littlewood--Paley decomposition and the symbol of \eqref{0.0}
\begin{equation}
\LL(i\tau, i\k, v) \stackrel{\text{def}}{=} i(\tau + \abf(v)\cdot\k) + \k\cdot\bbf(v)\k. \label{2.defL}
\end{equation}
Consequently, in order to ensure the convergence of such expansions, it was necessary to impose uniform decay rates on the quantities
\begin{equation}
\omega(J; \d) = \sup_{\sqrt{\tau^2 + |\k|^2}\sim J} \meas\Big\{ v \in \supp \eta; |\LL(i\tau, i\k, v)| \leq\d \Big\}. \label{0.2}
\end{equation}
By carefully studying the $L^r$--norm of these parcels, one could then verify the $W^{s,r}$--regularity of the averages \eqref{0.-1}. This method was further expanded in a series of works by \textsc{B. Gess--M. Hofmanová} \cite{GH} (with applications to stochastic quasilinear degenerate hyperbolic-parabolic equations), \textsc{B. Gess--X. Lamy} \cite{GL} (studying a conservation law with sources), \textsc{B. Gess} \cite{Ge} and \textsc{B. Gess--J. Sauer--E. Tadmor} \cite{GST} (both of the latter establishing the optimal Sobolev regularity for the porous medium equation).

Despite the impressive power and elegance of such an approach, it is not without a few shortcomings. We enumerate some below.
\begin{itemize}
	\item First of all, except for some elementary examples, the examination of the quantities \eqref{0.2} is somewhat laborious and so far have led only to partial results. For instance, concerning the simple parabolic-hyperbolic equation in $\R_t\X\R_x\X\R_y$
	\begin{equation}
	\frac{\del \varrho}{\del t} + \frac{\del}{\del x} \bigg\{ \frac{1}{\ell+1} \varrho^{\ell+1} \bigg\} - \frac{\del^2}{\del y^2} \bigg\{   \frac{1}{n+1} |\varrho |^{n+1} \bigg\} = 0, \label{0.3}
	\end{equation}
	where $n$ and $\ell$ are positive integers, their theory has as yet been shown to be applicable under the restriction that $n \geq 2\ell$ (see \textsc{E. Tadmor--T. Tao} \cite{TT}).
	\item As the behavior of the symbol $\LL(i\tau, i\k, v)$ is only treated obliquely via the quantities \eqref{0.2}, it is not clear which class of tempered distributions $\Lambda(t,x,v)$ is admissible in the right-hand side of \eqref{0.0}. (An exception to this is the recent work of \textsc{B. Gess--J. Sauer--E. Tadmor} \cite{GST}, who do not avoid the nonhomogeneity of $\LL(i\tau, i\k, v)$. Nevertheless, their method is quite complex and restricted to a particular type of equation.)
	\item Likewise, it is not clear if the \textsc{Tadmor--Tao} theory permits the diffusion matrix $\bbf(v)$ to degenerate on intervals, allowing Equation \eqref{0.0} to display a hyperbolic and a parabolic phase. This hypothesis is not of complete superficiality, as it appears naturally in applications to sedimentation-consolidation processes (see \textsc{M. C. Bustos} \textit{et al.} \cite{BCBT}).
\end{itemize}

The purpose of this manuscript is to present a new theory of averaging lemmas which overcomes the difficulties previously listed. The most interesting features of our method include:
\begin{enumerate}
	\item [(i)] The nondegeneracy conditions we consider are inspired by those introduced by \textsc{P.-L. Lions--B. Perthame--E. Tadmor} \cite{LPT}, once they are variants of
	\begin{align}
	\textnormal{``}\meas\big\{ v \in \supp \eta;\, \LL(i\tau, i\k, v) = 0 \big\} = 0 &\textnormal{ for all $(\tau, \k) \in \R\X\R^N$} \nonumber \\&\textnormal{ with $\tau^2 + |\k|^2 = 1$''.} \label{0.4}
	\end{align}
	As a consequence, they are of substantially easier verification;
	\item [(ii)] The distributions $\Lambda(t,x,v)$ appearing in \eqref{0.0} are allowed to have the form $\mathscr E(-\Delta_v + 1)^{\ell/2}g$, where $\ell \geq 0$, $g \in L^q(\R_t\X\R_x^N\X\R_v)$ ($1<q<\infty$), and $\mathscr E$ is an elliptic operator that ``tightly'' dominates $\LL\big(\frac{\del}{\del t}, \nabla_x, v\big)$. In particular, they can always involve full spatio-temporal derivatives of $g$, and they may contain second-order spatial derivatives of $g$ if $\LL\big(\frac{\del}{\del t}, \nabla_x, v\big)$ is parabolic for that particular velocity $v$, hence the ``criticality'' of our averaging lemmas. Accordingly, one gains an ample notion of the regularizing properties associated to the averaging process $f \mapsto \int_\R f\eta\,dv$;
	\item [(iii)] The proof is straightforward and transparent. Indeed, our arguments are based in the elementary method of \textsc{H. Frid} \textit{et al.} \cite{FL1} (see also \textsc{G.-Q. Chen--H. Frid}  \cite{CF0'}, and the averaging lemma 2.1 in \textsc{E. Tadmor--T. Tao}), but they contain refinements in every aspect. 
\end{enumerate}
Although the velocity averaging lemmas we will introduce are regarding the relative compactness of the averages \eqref{0.-1}, their reasoning may be altered so as to analyze the quantitative regularity of these averages as well. Thus, due to their simplicity, we reckon that our ideas may also elucidate the intricacies of the previous works on the subject.

\begin{remark}
	The results of this text were mainly motivated by the problem of proving the strong trace property for entropy solutions to stochastic parabolic-hyperbolic equations closely resembling \eqref{0.3}. This problem was successfully solved with \textsc{H. Frid} \textit{et al.} \cite{FL2} via the techniques of this manuscript (see also the revised version of \textsc{H. Frid--Y. Li} \cite{FL1}). Moreover, other applications of our averaging lemmas were also derived: in \cite{N1}, we introduced a method of establishing the well-posedness to initial-boundary value problems for stochastic scalar conservation laws, and we also investigated the Sobolev regularity of the solutions obtained therein; and, in \cite{N2}, we deduced simple criteria for the relatively compactness of entropy and kinetic solutions to deterministic convection-diffusion equations. Several other applications are to be explored in the future.
\end{remark}

\section{Main results}

\subsection{An illustrative example} \label{0.example}
Before properly stating our theorems, it is convenient to briefly look into a unidimensional model which not only explains our hypotheses, but also portrays the general principles behind our theory.

Suppose that $N = 1$, and, for all $n \in \N$, the equation
\begin{equation}
\frac{\del f_n}{\del t} + v \frac{\del f_n}{\del x} - \frac{\del}{\del x} \bigg( \bbf(v) \frac{\del f_n}{\del x} \bigg) = (-\Delta_{t,x})^{1/2}\frac{\del g_n}{\del v} \label{0.6.5}
\end{equation}
is satisfied in $\mathscr D'(\R_t\X\R_x\X\R_v)$, where $(f_n)_{n \in \N}$ is a bounded sequence in $L^2(\R_t\X\R_x\X\R_v)$, $(g_n)_{n \in \N}$ converges to zero in $L^2(\R_t\X\R_x\X\R_v)$, and $\bbf : \R \to \R$ is a smooth, nonnegative function. Our desire is to show that, given any weight function $\eta \in \Cc_c^\infty(\R_v)$, the averages $\int_{\R} f_n \eta\, dv $ are relatively compact in $L_\loc^2(\R_t\X\R_x^N\X\R_v)$.

Notice that one may assume that $f_n \rightharpoonup f$ weakly in $\s(L_{t,x,v}^2; L_{t,x,v}^2)$; in this case, the weak limit $f(t,x,v)$ surely obeys the equation $$\frac{\del f}{\del t} + v \frac{\del f}{\del x} - \frac{\del}{\del x} \bigg(\bbf(v)\frac{\del f}{\del x} \bigg) = 0.$$ Since $f \in L^2(\R_t\X\R_x^N\X\R_v)$, one may apply the classical techniques of Sobolev spaces and semigroups of operators to deduce that $f \equiv 0$ in the $L^2$--sense. As a result, it is clear that, if $\int_{\R} f_n \eta\, dv $ is relatively compact, then it converges \textit{a fortiori} to $0$ in $L_\loc^2$.

The traditional argument in the theory of the averaging lemmas is roughly as follows. If $\FF_{t, x}$ denotes the Fourier transform in $(t,x)$, it can be seen that
$$\big( i(\tau + v\k) + \bbf(v)\k^2 \big) (\FF_{t, x} f_n) = \sqrt{\tau^2 + |\k|^2}\frac{\del}{\del v} (\FF_{t, x} g_n).$$
This formula is very meaningful if $\LL(i\tau, i\k, v) = i(\tau + v\k) + \bbf(v)\k^2$ is not too small, as one may then formally divide the equation by $\LL(i\tau, i\k, v)$. In order to discern when $\LL(i\tau, i\k, v)$ is acceptably far away from zero, let $(\tau', \k')$ denote the normalized frequency
\begin{equation}
(\tau', \k') = \frac{1}{\sqrt{\tau^2 + |\k|^2}}(\tau,\k) \label{2.deftauk'}
\end{equation}
for $(\tau, \k) \neq 0$, and introduce some $\psi \in \Cc^\infty_c(\C; \R)$ such that $\psi(z) = 0$ for $|z| < 1/2$ and $\psi(z) = 1$ for $|z| > 1$. Then, for any $0 < \g$ and $\d < 1$, one may decompose $f_n$
as
$$f_n \stackrel{\text{def}}{=} f_n^{(1)} + f_n^{(2)} + f_n^{(3)},$$
where
\begin{align*}
(\FF_{t,x} f_n^{(1)})(\tau, \k, v) &\stackrel{\text{def}}{=} (1-\psi)\bigg( \frac{\sqrt{\tau^2 + |\k|^2}}{\g} \bigg) (\FF_{t,x} f_n)(\tau, \k, v), \\
(\FF_{t,x} f_n^{(2)})(\tau, \k, v) &\stackrel{\text{def}}{=} \psi\bigg( \frac{\sqrt{\tau^2 + |\k|^2}}{\g} \bigg) (1- \psi)\bigg( \frac{\LL(i\tau', i\k', v)}{\d} \bigg) (\FF_{t,x} f_n)(\tau, \k, v), \text{ and} \\
(\FF_{t,x} f_n^{(3)})(\tau, \k, v) &\stackrel{\text{def}}{=} \psi\bigg( \frac{\sqrt{\tau^2 + |\k|^2}}{\g} \bigg)  \psi \bigg( \frac{\LL(i\tau', i\k', v)}{\d} \bigg) (\FF_{t,x} f_n)(\tau, \k, v).
\end{align*}
One may interpret this division as follows. $f_n^{(1)}$ is formed by the low-frequencies of $f_n$, wherefore it is naturally well-behaved (recall, for instance, the Paley--Wiener theorem). On the other hand, $f_n^{(2)}$ is the part of $f_n$ that is supported where $|\LL(i\tau', i\k', v)|$ is small, and thus its average may be uniformly handled thanks to the nondegeneracy condition \eqref{0.4} (hence the necessity of such a hypothesis).  Observe that $\LL(i\tau, i\k, v)$ verily satisfies \eqref{0.4}, for its hyperbolic part $(\tau, \k, v) \mapsto i(\tau + v\k)$ certainly does.

At last, the remainder term, $f_n^{(3)}$, is the parcel of $f_n$ located in the high frequencies for which $|\LL(i\tau', i\k', v)| \geq \d/2$. Therefore, it may be analyzed through the differential equation \eqref{0.6.5}, in the sense that
\begin{align}
(\FF_{t,x} f_n^{(3)}) =  \psi\bigg( \frac{\sqrt{\tau^2 + |\k|^2}}{\g} \bigg)\psi\bigg( \frac{\LL(i\tau', i\k', v)}{\d} \bigg) \frac{\sqrt{\tau^2 + |\k|^2}}{\LL(i\tau, i\k,v)}\frac{\del}{\del v} (\FF_{t, x} g_n). \label{0.6}
\end{align}
As we argued, this is the sole element one should be preoccupied with, consequently we will only pay attention to it for now. Multiplying \eqref{0.6} by $\eta(v)$ and integrating in $v \in \R_v$ implies that
\begin{align}
\FF_{t,x} \bigg(& \int_\R f_n^{(3)} \eta\, dv \bigg) = \nonumber \\ & - \int_{\R}\psi\bigg( \frac{\sqrt{\tau^2 + |\k|^2}}{\g} \bigg) \frac{\del}{\del v} \bigg\{  \eta(v) \psi\bigg( \frac{\LL(i\tau', i\k', v)}{\d} \bigg) \bigg\} \frac{\sqrt{\tau^2 + |\k|^2}}{\LL(i\tau, i\k, v)} (\FF_{t, x} g_n)\, dv \nonumber \\ & - \int_\R \eta(v) \psi\bigg( \frac{\sqrt{\tau^2 + |\k|^2}}{\g} \bigg) \psi\bigg( \frac{\LL(i\tau', i\k', v)}{\d} \bigg) \frac{\del}{\del v} \bigg\{   \frac{\sqrt{\tau^2 + |\k|^2}}{\LL(i\tau, i\k, v)} \bigg\}  (\FF_{t, x} g_n)\, dv. \label{0.7}
\end{align}
On the strength of the Plancherel theorem, and the Cauchy--Schwarz inequality,
$$\int_{\R_{t,x}^2}\bigg|\int_{\R_v} \Lambda(t,x,v) \phi(v)\,dv \bigg|^2\,dxdt \leq \bigg(\int_{\R_v} |\phi(v)|^2\,dv\bigg)\bigg(\int_{\R_{t,x,v}^3} |\Lambda(t,x,v)|^2\,dvdxdt \bigg) $$
and the assumption that $g_n \to 0$ in $L_{t,x,v}^2$, it is not difficult to see that, so as to guarantee that $\int_\R f_n^{(3)} \eta \, dv$ also converges to $0$, it suffices to establish that
\begin{align} \label{0.8}
\begin{dcases}
\sqrt{\tau^2 + |\k|^2} \leq C |\LL(i\tau, i\k, v) |, \text{ and} \\
|\LL_v(i\tau, i\k, v) | \leq C |\LL(i\tau, i\k, v)|
\end{dcases}
\end{align}
where $\LL_v(i\tau, i\k, v) = \frac{\del \LL}{\del v}(i\tau, i\k, v)$, $(\tau, \k) \in \mathscr B(v) = \big\{ |\LL(i\tau', i\k', v)| \geq \d/2 \} \cap \{ \sqrt{\tau^2 + |\k|^2} \geq \g/2\big\}$, and $v \in \supp \eta$.

Due to the restriction $(\tau, \k) \in \mathscr B(v)$, the first inequality \eqref{0.8} follows quite easily. Moreover, if $\bbf(v) \equiv 0$ for such $v$'s (i.e., the equation is hyperbolic in the support of $\eta$), the second inequality is equally trivialized, for it then becomes a relation between two homogeneous functions of degree $0$.

On the other hand, if $\bbf(v) \not \equiv 0$, the second desired estimate becomes much more delicate. For the sake of the argument, let us assume that $\bbf(v) = v^2$, so that $\LL_v(i\tau, i\k, v) = iv\k + 2v \k^2$. Thus, choosing $(\tau', \k')$ such that $\tau'$ is very close to $1$ (forcing $|\LL(i\tau', i\k', v)|$ to be close to $1$ as well), and $v \neq 0$, one can infer that
$$\sup_{(\tau, \k) \in \mathscr B(v)}\bigg| \frac{\LL_v(i\tau, i\k, v)}{\LL(i\tau, i\k, v)}\bigg| \geq \frac{2}{|v|},$$
which becomes very singular when $v$ approaches the origin. As a corollary, \eqref{0.8} is not feasible if $0 \in \supp \eta$. 

Nevertheless, this complicating velocity is just a single point. Thus, one can truncate the weight function $\eta$ near it, and indeed \eqref{0.8} would hold. The residual term, composed by the velocities neighboring $0$, can be made uniformly small due to $L^2$--boundness of $f_n$. As a consequence, one may argue that $\int_\R f_n \eta\, dv \to 0$ in $L_\loc^2$, as we wanted to show.

The issue above and its resolution indicate that solely employing the quantity $\LL(i\tau', i\k', v)$ may not be adequate to measure the degeneracy of equation \eqref{0.6.5} when $\bbf(v) \not\equiv 0$. In reality, the heart of the matter in the parabolic case is not that one should select the non-degenerate directions of $\LL(i\tau, i\k, v)$, but that one should ensure that $\LL(i\tau, i\k, v)$ behaves like the heat equation symbol $\mathcal C(\tau, \k) = i\tau + |\k|^2$. If this property is secured, not only can one bound $\LL_v(i\tau, i\k, v)$, but also one may then control a stronger operator than $(-\Delta_{t,x})^{1/2}$: one may indeed substitute $(-\Delta_{t,x})^{1/2}$ for $(-\Delta_{t,x} + 1)^{1/2} - \Delta_x$, an elliptic operator that ``tightly'' dominates $\mathcal C(\tau, \k)$.

Furthermore, this toy model also suggests the following method for investigating \eqref{0.6.5} with a general $\bbf(v)$. One separates $\R_v$ into two subsets: the one where $\bbf(v) \equiv 0$ identically, and the one where $\bbf(v) > 0$. In the former, one can apply the simple argument of the hyperbolic case, whereas, in the latter, provided that one remains bounded away from $\{\bbf(v) = 0\}$, the argument for $\bbf(v) = v^2$ would hold fine. Then, assuming that the set where \eqref{0.6.5} mutates from a ``hyperbolic'' phase to a ``parabolic'' one---or vice versa---is ``small'', this agglutination would recover the complete average $\int_\R f_n \eta \, dv$, thence showing its convergence to $0$ in $L_\loc^2$. Theorems \ref{1.firstthm} and \ref{1.secondthm} of this paper investigate this reasoning.

Notwithstanding, if $\bbf(v)$ does not degenerate in entire intervals but only in null sets (as, e.g., $\bbf(v) = v^2$), a considerably better manner to evaluate the behavior of $\LL(i\tau, i\k, v)$ would be to employ
\begin{equation}
\psi\bigg( \frac{\text{real part of } \LL(i\tau, i\k, v)}{\d|\k|^2} \bigg) = \psi\bigg(\frac{\bbf(v)}{\d}\bigg), \label{0.9}
\end{equation}
as this function deftly measures the diffuseness of $\LL(i\tau, i\k, v)$. Notice that, when $\bbf(v) = v^2$, $\psi(\bbf(v)/\d)$ only truncates the velocities near $0$, exactly as we have argued before. This hypothesis on the set of degeneracy of $\bbf(v)$, which fundamentally says that Equation \ref{0.6.5} possesses one unique regime (as opposed to the previous scenario), is considered in depth in Theorems \ref{1.firstparthm} and \ref{1.secondparthm}.

One central matter we have not touched upon above is the extension from $L^2$ to a general $L^p$--space for $1 < p < \infty$. This is a quite dramatic paradigm shift, as the Plancherel theorem is unavailable and thus the simple conditions \eqref{0.8} are no longer enough to prove that $\int_\R f_n^{(3)} \eta\, dv$ converges in $L^p$. Consequently, one is forced to apply multiplier theorems in order to analyze such averages; however, most $L^p$--multiplier theorems, such as the celebrated result of Mihlin--Hörmander, are not well-suited to examine functions like
\begin{equation}
\psi\bigg(\frac{\sqrt{\tau^2 +|\k|^2}}{\g}\bigg) \psi\bigg(\frac{\bbf(v)}{\d}\bigg) \frac{\LL_v(i\tau, i\k, v)}{\LL(i\tau, i\k, v)} \label{0.10}
\end{equation}
in virtue of its lack of homogeneity for large $\sqrt{\tau^2+ |\k|^2}$. Fortunately, there exists a criterion due to \textsc{P. I. Lizorkin} \cite{L}, whose usage seems to have been so far restricted to the Fourier analysis in abstract Banach spaces, that neatly facilitates the investigation of anisotropic multipliers such as \eqref{0.10}. In this way, the principles we have just portrayed can be extended $L^p$, which is truly the case of interest in nonlinear problems.

\subsection{The statement of the main results}
With this philosophy in mind, let us determine some notations and hypotheses.

Inspired by the previous work of \textsc{B. Gess--M. Hofmanová} \cite{GH}, we will also consider certain stochastic terms in the right-hand side of \eqref{0.0}; even so, if one is interested in purely deterministic results, one only needs to let the $\Phi_n$'s appearing henceforth to be $0$. In any event, our probabilistic framework is as follows. The triplet $(\Omega, \Ff, \bbP)$ will stand for a probability space endowed with a complete, right-continuous filtration $(\Ff_t)_{t \geq 0}$. Furthermore, it will be assumed the existence of a sequence $(\b_k(t))_{k \in \N}$ of mutually independent Brownian motions in $\left(\Omega, \Ff, (\Ff_t)_{t \geq 0}, \bbP\right)$, so that, if $\mathscr{H}$ is a separable Hilbert space with a hilbertian basis $(e_k)_{k \in \N}$, $W(t) = \sum_{k=1}^\infty \b_k(t) e_k$ defines a cylindrical Wiener process. Recall that, if $\mathfrak U$ is another separable Hilbert space, $HS(\mathscr H; \mathfrak U)$ denotes the space of the Hilbert--Schmidt operators $T \in \mathscr L(\mathscr H; \mathfrak U)$.

Let $N \geq 1$ be an integer. The next definitions are central to the theory here exposed.

\begin{definition} \label{1.defb}
	Let $\bbf : \R \to \mathscr{L}(\R^N)$ be a nonnegative matrix function.
	\begin{enumerate}
		\item $\bbf$ is said to have a \textit{dichotomous range} if there exists a fixed linear subspace $M \subset \R^N$ such that, for every $v \in \R$, $R(\bbf(v))$, the range of $\bbf(v)$, is either $M$ or $\{0\}$. The maximal subspace $M$ for which such an alternative holds is called the \textit{effective range} of $\bbf$.
		\item $\bbf$ is said to satisfy the \textit{nontransiency condition} in a given measurable set $X \subset \R$ if, putting $F$ to be the boundary of $\{ v \in \R; \bbf(v) = 0 \}$, $F \cap X$ is a null set with respect to the Lebesgue measure.
	\end{enumerate}
\end{definition}

\begin{remark}
	The nontransiency condition translates quantitatively the notion that the set of velocities in which \eqref{0.6.5} passes from a parabolic regime to a hyperbolic one, or vice versa, is small. On the other hand, the effective range hypothesis allows one to generalize the syllogism of Subsection \ref{0.example} to multidimensional anisotropic equations.
\end{remark}

Finally, recall that, given any linear subspace $M \subset \R^N$, the Laplacean operator restricted to $M$ is defined as $$\Delta_M \stackrel{\text{def}}{=} \div_x(P_M \nabla_x),$$ where $P_M$ denotes the orthogonal projection onto $M$. Notice that, in terms of the Fourier transform, given any $\phi \in \Ss(\R_x^N)$, $$\FF_{x}( (-\Delta_M) \phi )(\k) = |P_M \k|^2 (\FF_x\phi)(\k).$$ Likewise, recollect that, given any matrix $\m = (\m_{\mu, \nu})_{1 \leq \mu,\nu \leq N} \in \mathscr L(\R^N)$, the differential operator $D_x^2 : \m$ is defined by $$D_x^2 : \m \stackrel{\text{def}}{=} \sum_{\mu, \nu = 1}^N \m_{\mu,\nu}\frac{\del^2}{\del x_\mu \del x_\nu} = \div_x ( \m \nabla_x ).$$

With these conventions in mind, let us enunciate our first velocity averaging lemma.

\begin{theorem} [The global ``two-phase'' averaging lemma] \label{1.firstthm}
	Let $\JJ$ be finite index set, and let be given exponents $1 < p, q_j < \infty$ $(j \in \JJ)$, $1 \leq r \leq 2$ and $\ell \geq 0$. Assume that $\abf \in \Cc_\loc^{k, \a}(\R; \R^N)$ and $\bbf \in \Cc_\loc^{k, \a}(\R; \mathscr{L}(\R^N))$, where the real numbers $k$ and $\a$ are such that
	\begin{equation} \label{1.condka}
	(k,  \a ) \in \begin{dcases}
	\{ 0 \} \X \{ 0 \} &\text{if $\ell = 0$,} \\
	\{ \lfloor \ell \rfloor \} \X (\ell - \lfloor \ell \rfloor , 1 ]  &\text{if $\ell > 0$ is not an integer, and}\\
	\{\ell - 1 \} \X \{ 1 \}   &\text{if $\ell \geq 1$ is an integer,}
	\end{dcases}
	\end{equation}
	and $\bbf(v)$ is nonnegative for all $v \in \R$ and has a dichotomous range. Let $M$ be the effective range of $\bbf$.
	
	Suppose that, for any integer $n \in \N$, the equation
	\begin{align}
	\frac{\del f_n}{\del t} &+ \abf(v) \cdot \nabla_x f_n - \bbf(v) : D_x^2 f_n = \sum_{j \in \JJ}( - \Delta_{t,x} + 1)^{1/2} (- \Delta_v + 1)^{\ell/2}  g_{j,n} \nonumber \\ & +\sum_{j\in \JJ}(\Pi_j(v) \Delta_M)(- \Delta_v + 1)^{\ell/2}  h_{j,n} + ( - \Delta_{x} + 1)^{1/4} (- \Delta_v + 1)^{\ell/2} \Phi_n \frac{dW}{dt} \label{1.eqf}
	\end{align}
	is almost surely obeyed in $\Dd'(\R_t\X\R_x^N\X\R_v)$, where
	\begin{enumerate}
		\item  $(f_n)_{n\in\N}$ is a bounded sequence in $L^r(\Omega; L^p(\R_t\X\R_x^N\X\R_v))$,
		\item for all $j \in \JJ$, $(g_{j,n})_{n \in \N}$ and $(h_{j,n})_{n \in \N}$ are relatively compact sequences in $L^r(\Omega; L^{q_j}$ $(\R_t$ $\X\,\,\R_x^N\X\R_v))$,
		\item for all $j \in \JJ$, $\Pi_j \in \Cc_\loc^{k, \a}(\R)$ is such that $\supp \Pi_j \subset \supp \bbf$, and
		\item  $(\Phi_n)_{n \in \N}$ is a predictable and relatively compact sequence in $L^2(\Omega\X [0,\infty)_t;$ $HS(\mathscr H;$ $L^2(\R_x^N \X \R_v)))$.
	\end{enumerate}

	Finally, let $\eta \in L^{p'}(\R)$ have compact support, and presume that the \textnormal{nondegeneracy condition}
	\begin{align} 
	\meas \big\{ v \in \supp \eta; \tau + \abf(v) \cdot \k = 0 &\text{ and } \k \cdot \bbf(v)\k = 0  \big\} = 0 \nonumber \\
	&\text{ for all $(\tau, \k) \in \R\X\R^N$ with $\tau^2 + |\k|^2 = 1$}\label{1.nondeg}
	\end{align}
	holds, and that $\bbf(v)$ satisfies the \textnormal{nontransient condition} in $\supp \eta$.
	
	Then, with $s$ being the least number between $p$, $q_j$ $(j \in J)$, and $2$, the sequence of averages $\left(\varphi \int_\R f_n\eta \, dv \right)_{n \in \N}$ is relatively compact in $L^r(\Om; L^s(\R_t\X\R_x^N))$  for any $\varphi \in  (L^{1} \cap L^\infty)(\R_t\X\R_x^N)$. 
\end{theorem}

Some observations are in order.

\begin{remark}[On the meaning of \eqref{1.eqf}] \label{1.remark}
	Conserving the assumptions of the first two paragraphs of Theorem \ref{1.firstthm}, the differential equation \eqref{1.eqf} should be understood as follows: Almost surely, it holds that
	\begin{align}
	-\int_{\R_t} \int_{\R_x^N} \int_{\R_v} &f_n \left( \frac{\del \phi}{\del t} + \abf(v) \cdot \nabla_x \phi + \bbf(v) : D_x^2 \phi\right)\, dv\,dx\,dt \nonumber\\ =&	\sum_{j \in \JJ}  \int_{\R_t} \int_{\R_x^N} \int_{\R_v} \left( (- \Delta_v + 1)^{\ell/2}(- \Delta_{t,x} + 1)^{1/2}\phi \right)  g_{j,n}  \,dvdxdt  \nonumber \\
	&+  \sum_{j \in \JJ} \int_{\R_t} \int_{\R_x^N} \int_{\R_v} \Big((- \Delta_v + 1)^{\ell/2} (\Pi_j(v) \Delta_M\phi) \Big)  h_{j,n}  \,dvdxdt \nonumber \\
	&+ \int_0^\infty \int_{\R_x^N} \int_{\R_v}  \Big((- \Delta_v + 1)^{\ell/2} ( - \Delta_{x} + 1)^{1/4}\phi \Big)  \Phi_n \, dvdxdW(t) \label{1.eqftil}
	\end{align}
	for all $\phi \in \Cc_c^\infty(\R_t\X\R_x^N\X\R_v)$ and $n \in \N$. Due to the Hölder regularity of the $\Pi_j$'s and the compact support of $\phi$, each and every term in \eqref{1.eqftil} is almost surely well-defined---see, e.g., Proposition \ref{1.7.prop}. Clearly, this definition may be extended to the case in which, rather than in the entire space $\R_t\X\R_x^N$, one is only considering $(t,x)$ lying in some smaller open set $Q \subset \R_t\X\R_x^N$.
\end{remark}

\begin{remark}[On the linear subspace $M$]
	Certainly, one could have assumed without loss of generality that $M$ had the form
	$$M = \big\{ x = (x_1, \ldots, x_N) \in \R^N; x_\nu = 0 \text{ for $N' < \nu$} \big\},$$
	where $N' = \dim M$ is a fixed integer. In this case, $\Delta_M$ would be simply $$\Delta_M = \frac{\del^2}{\del x_1^2} + \cdots + \frac{\del^2}{\del x_{N'}^2}.$$ Nevertheless, we have opted not to do so, as we reckon this would significantly clutter the notation. Anyhow, the linear subspace $M$ is introduced in order to consider equations that are only parabolic in some variables (such as \eqref{0.3}).
\end{remark}

\begin{remark}[On the set $\JJ$, the functions $\Pi_j(v)$, etc]
	Essentially, $\Pi_j(v)$ are present in order that the deterministic source terms in \eqref{1.eqf} to carry full second--order derivatives in $x$ during the ``parabolic'' phase of \eqref{1.eqf}, ascertaining the criticality of Theorem \ref{1.firstthm}. In accordance to our previous discussion, notice that, if $R(\bbf(v)) = M$, then  $(-\Delta_{t,x}+1)^{1/2} + (-\Delta_M)$ is the elliptic that tightly dominates $\LL(i\tau, \k, v)$.
	
	So as to be more consistent with this philosophy, the right-hand side of \eqref{1.eqf} could have also included terms of the form
	\begin{equation}
	\sum_{j \in \JJ} \Upsilon_j(v)  (-\Delta_M)^{1/2} (-\Delta_v + 1)^{\ell/2} \Psi_{j,n}, \label{0.11}
	\end{equation}
	where, for any $j \in \JJ$, $\Upsilon_j \in \Cc_\loc^{k,\a}(\R)$ with $\supp \Upsilon_j \subset \supp \bbf$, and $(\Psi_{j,n})_{n \in \N}$ is predictable and relatively compact in $L^2(\Om\X[0,\infty); HS(\R_x^N\X\R_v))$. Indeed, it is well--known that solutions to stochastic differential equations involving the white noise possess one--half of the spatial regularity one would expect from their deterministic counterparts. Nevertheless, we will omit such terms like \eqref{0.11} for simplicity's sake.
	
	Let us mention that, in spite of the index set $\JJ$ commonly being a singleton, it is important to let $\JJ$ be a general finite set so that \eqref{1.eqf} is ``closed under localizations''---see the next theorem.
\end{remark}

Even though the next averaging lemma is derivative of the former, its statement is better adapted to some applications. Again, let us first fix another notation.

Let $1 \leq p \leq \infty$, $\mathscr E$ be an Euclidean space, and $\mathscr U \subset \mathscr E$ be an open set. $L^r(\Om; L_\loc^p(\mathscr U))$ will represent the set of all mappings $f : \Om \to L_\loc^{p}(\mathscr U)$, such that $\theta f \in L^r(\Om; L^{p}(\mathscr U))$ for any $\theta \in \Cc_c^\infty(\mathscr U)$. This set clearly exemplifies the notion of a Fr\'{e}chet space. 

\begin{theorem} [The local ``two-phase'' averaging lemma] \label{1.secondthm}
	Let $\JJ$ be finite index set, and let be given exponents $1 < p, q_j < \infty$ $(j \in \JJ)$, $1 \leq r \leq 2$ and $\ell \geq 0$. Assume that $\abf \in \Cc_\loc^{k, \a}(\R; \R^N)$ and $\bbf \in \Cc_\loc^{k, \a}(\R; \mathscr{L}(\R^N))$, where the real numbers $k$ and $\a$ satisfy the relation \eqref{1.condka}, and $\bbf(v)$ is nonnegative for all $v \in \R$ and has a dichotomous range. Moreover, let $M$ be the effective range of $\bbf$, and let $Q \subset \R_t\X\R_x^N$ be an open set.
	
	Suppose that, for any $n \in \N$, the equation \eqref{1.eqf} is almost surely obeyed in $\Dd'(Q\X\R_v)$, where
	\begin{enumerate}
		\item  $(f_n)_{n\in\N}$ is a bounded sequence in $L^r(\Omega; L_\loc^p(Q\X\R_v))$ such that, for any $\phi \in \Cc_c^\infty(Q)$ and $\z \in \Cc_c^\infty(\R_v)$, either
		\begin{enumerate}
			\item both $\begin{dcases*}
			(-\Delta_v + 1)^{-\ell/2} (-\Delta_{t,x}+1)^{-1/2} (\phi \z f_n) \text{ and}\\
			(-\Delta_v + 1)^{-\ell/2} (-\Delta_{t,x} + 1)^{-1/2} (\nabla_x \phi \cdot \,\operatorname{div}_x(\z f_n \bbf ))
			\end{dcases*}$ are relatively compact in $L^r(\Omega; L^p(\R_t\X\R_x^N\X\R_v))$, or
			\item $(-\Delta_v + 1)^{-\ell/2} (-\Delta_{t,x} + 1)^{-1/4} (\phi f_n)$ is relatively compact in $L^r(\Omega; L^p(\R_t\X\R_x^N\X\R_v))$, 
		\end{enumerate}
		\item for all $j \in \JJ$, $(g_{j,n})_{n \in \N}$ and $(h_{j,n})_{n \in \N}$ are relatively compact sequences in $L^r(\Omega; L^{q_j}$ $(\R_t\X\R_x^N\X\R_v))$,
		\item for all $j \in \JJ$, $\Pi_j \in \Cc_\loc^{k, \a}(\R)$ is such that $\supp \Pi_j \subset \supp \bbf$, and
		\item  $(\Phi_n)_{n \in \N}$ is a predictable and relatively compact sequence in $L^2(\Omega\X [0,\infty)_t;$ $HS(\mathscr H;$ $L^2(\R_x^N \X \R_v)))$.
	\end{enumerate}
	
	Finally, let $\eta \in L^{p'}(\R)$ have compact support, and presume that the \textnormal{nondegeneracy condition} \eqref{1.nondeg} holds, and that $\bbf(v)$ satisfies the \textnormal{nontransient condition} in $\supp \eta$.
	
	Then, the sequence of averages $(\int_\R f_n\eta \, dv)_{n \in \N}$ is relatively compact in $L^r(\Om; L_\loc^s(Q))$, with $s$ being the least number between $p$, $q_j$ $(j \in \JJ)$, and $2$. In particular, if  $(f_n)_{n\in\N}$ is bounded in $L^r(\Omega; L^p(Q\X \supp \eta))$, and $Q$ is of finite measure, the averages $\left(\int_\R f_n\eta \, dv \right)_{n \in \N}$ are relatively compact in $L^r(\Om; L^z(Q))$ for any $1 \leq z < p$.
\end{theorem}

\begin{remark} [On the conditions (a) and (b)]
	In the probabilistic setting we are considering, it is pivotal to impose the relative compactness of $(f_n)$ in a local, anisotropic, negative Sobolev space, once this would not be a corollary of weak convergence arguments as it would have been in the deterministic case. Although such conditions do not hold in general, there exist certain procedures involving the Prohorov compactness theorem, the Skohorod representation theorem, and the Gyöngi--Krylov lemma which allow such hypotheses; see, e.g., \textsc{A. Debussche--M. Hofmanová--J. Vovelle} \cite{DHV}, \textsc{H. Frid} \textit{et al.} \cite{FL1}, and the references therein.
	
	Once more, if one is solely concerned with deterministic results, one may disregard both conditions (a) and (b). Further comments on such conditions are postponed to Section \ref{1.secremarks}.
\end{remark}

	We now turn to the averaging lemmas for equations which display one specific behavior. We notice that, under such a circumstance, the statements of the results are quite facilitated.
	
	\begin{theorem} [The global ``single-phase'' averaging lemma] \label{1.firstparthm}
		Let $\JJ$ be finite index set, and let be given exponents $1 < p, q_j < \infty$ $(j \in\JJ)$, $1 \leq r \leq 2$ and $\ell \geq 0$. Assume that $\abf \in \Cc_\loc^{k, \a}(\R; \R^N)$ and $\bbf \in \Cc_\loc^{k, \a}(\R; \mathscr{L}(\R^N))$, where the real numbers $k$ and $\a$ satisfy the relation \eqref{1.condka}. Furthermore, suppose that there exists a linear subspace $M \subset \R^N$, such that $R(\bbf(v)) \subset M$ and $\bbf(v)$ is nonnegative for all $v \in \R$. 
		
		Assume that, for any $n \in \N$, the equation
		\begin{align}
			\frac{\del f_n}{\del t} + \abf(v) \cdot \nabla_x f_n - \bbf(v) : D_x^2 &f_n =	  \sum_{j \in \JJ} \big(( - \Delta_{t,x} + 1)^{1/2} - \Delta_M\big)  (- \Delta_v + 1)^{\ell/2}  g_{j,n} \nonumber \\ &+ \big( (-\Delta_x + 1)^{1/4} + (-\Delta_M)^{1/2} \big) (- \Delta_v + 1)^{\ell/2} \Phi_n \frac{dW}{dt} \label{1.eqfpar}
		\end{align}
		is almost surely obeyed in $\Dd'(\R_t\X\R_x^N\X\R_v)$, where
		\begin{enumerate}
			\item  $(f_n)_{n\in\N}$ is a bounded sequence in $L^r(\Omega; L^p(\R_t\X\R_x^N\X\R_v))$,
			\item for all $j \in \JJ$, $(g_{j,n})_{n \in \N}$ is a relatively compact sequence in $L^r(\Omega; L^{q_j}(\R_t$ $\X\,\,\R_x^N\X\R_v))$, and
			\item  $(\Phi_n)_{n \in \N}$ is a predictable and relatively compact sequence in $L^2(\Omega\X [0,\infty)_t;$ $HS(\mathscr H;$ $L^2(\R_x^N \X \R_v)))$.
		\end{enumerate}

		Finally, let $\eta \in L^{p'}(\R)$ have compact support, and presume that the \textnormal{nondegeneracy condition}
		\begin{align} 
		\meas \big\{ v \in \supp \eta; \tau + (P_{M^\perp} &\abf)(v) \cdot \k = 0 \text{ and } \k \cdot \bbf(v)\k = 0  \big\} = 0 \nonumber\\
		&\text{ for all $(\tau, \k) \in \R\X\R^N$ with $\tau^2 + |\k|^2 = 1$}\label{1.nondegp}
		\end{align}
		holds.
		
		Then, with $s$ being the least number between $p$, $q_j$ $(j \in \JJ)$, and $2$, the sequence of averages $\left(\varphi \int_\R f_n\eta \, dv \right)_{n \in \N}$ is relatively compact in $L^r(\Om; L^s(\R_t\X\R_x^N))$  for any $\varphi \in  (L^{1} \cap L^\infty)(\R_t\X\R_x^N)$. 
	\end{theorem}

\begin{remark}[On the nondegeneracy condition \ref{1.nondegp}]
	In a nutshell, the nondegeneracy condition \eqref{1.nondegp} forces that the symbol $(\tau, \k, v) \mapsto i(\tau + (P_{M^\perp}\abf)(v)\cdot\k) + \k \cdot \bbf(v)\k$ to obey the usual imposition \eqref{1.nondeg}, thus exempting  any restriction on $(P_{M} \abf)(v)$ (the component of $\abf(v)$ which acts on the ``parabolic'' variables). In accordance to the particular behavior of \eqref{1.eqfpar}, the usage of the localizing functions $\Pi_j$ could be dispensed.
\end{remark}

	Let us also state a local version of the previous theorem.

\begin{theorem} [The local ``single-phase'' averaging lemma] \label{1.secondparthm}
	Let $\JJ$ be finite index set, and let be given exponents $1 < p, q_j < \infty$ $(j \in\JJ)$, $1 \leq r \leq 2$ and $\ell \geq 0$. Assume that $\abf \in \Cc_\loc^{k, \a}(\R; \R^N)$ and $\bbf \in \Cc_\loc^{k, \a}(\R; \mathscr{L}(\R^N))$, where the real numbers $k$ and $\a$ satisfy the relation \eqref{1.condka}. Furthermore, suppose that there exists a linear subspace $M \subset \R^N$, such that $R(\bbf(v)) \subset M$ and $\bbf(v)$ is nonnegative for all $v \in \R$. Let $Q \subset \R_t\X\R_x^N$ be an open set.
	
	Assume that, for any $n \in \N$, Equation \eqref{1.eqfpar} is obeyed in $\mathscr D'(Q\X\R_v)$, where
	\begin{enumerate}
		\item  $(f_n)_{n\in\N}$ is a bounded sequence in $L^r(\Omega; L_\loc^p(Q\X\R_v))$, such, for any $\phi \in \Cc_c^\infty(Q)$ and $\z \in \Cc_c^\infty(\R)$, $(-\Delta_v + 1)^{-\ell/2}(1 + (-\Delta_{t,x})^{1/2} - \Delta_M)^{-1}(\phi \z f_n)$ is relatively compact in $L^r(\Omega; L^p(\R_t\X\R_x^N\X\R_v))$,
		\item for all $j \in \JJ$, $(g_{j,n})_{n \in \N}$ is a relatively compact sequence in $L^r(\Omega; L^{q_j}(\R_t$ $\X\,\,\R_x^N\X\R_v))$, and
		\item  $(\Phi_n)_{n \in \N}$ is a predictable and relatively compact sequence in $L^2(\Omega\X [0,\infty)_t;$ $HS(\mathscr H;$ $L^2(\R_x^N \X \R_v)))$.
	\end{enumerate}
	
	Finally, let $\eta \in L^{p'}(\R)$ have compact support, and presume that the \textnormal{nondegeneracy condition} \eqref{1.nondegp} holds.
	
	Then, with $s$ being the least number between $p$, $q_j$ $(j \in \JJ)$, and $2$, the sequence of averages $\left(\int_\R f_n\eta \, dv \right)_{n \in \N}$ is relatively compact in $L^r(\Om; L_\loc^s(Q))$. In particular, if $(f_n)_{n\in\N}$ is bounded in $L^r(\Omega; L^p(Q\X\supp \eta))$, and $Q$ is of finite measure, then $(\int_\R f_n\eta \, dv )_{n \in \N}$ converges in $L^r(\Om; L^z(Q))$ for any $1 \leq z < p$.
\end{theorem}

\begin{remark} [On the hypotheses on $\bbf(v)$]
	In the theory of flow in porous media, the matrix $\bbf(v)$ only degenerates in a single point. Therefore, $\bbf(v)$ evidently obeys the nontransiency condition, and both lines of theorem apply, even though Theorems \ref{1.firstparthm} and \ref{1.secondparthm} are likely preferable. On the other hand, in sedimentation-consolidation processes, $\bbf(v)$ has the isotropic form
	\begin{equation}
	\bbf(v) = \mathbf q(v) I_{\R^N}, \label{0.bsedimentation}
	\end{equation}
	with $\mathbf q : \R \to \R$ satisfying $\mathbf q(v) > 0$ in some interval $I$, and $\mathbf q(v) = 0$ outside of $I$. Clearly again $\bbf(v)$ observes the nontransiency condition, and Theorems \ref{1.firstthm} and \ref{1.secondthm} are available. 
	
	
	On a more theoretical note, let us point out that, in contrast with Theorems \ref{1.firstthm} and \ref{1.secondthm}, it is permissible that $R(\bbf(v)) \neq M$ everywhere. By way of illustration, if $N=2$ and $M = \R^2$,
	$$\bbf(v) = \begin{pmatrix}
	v^2 & v^3\\
	v^3 & v^4
	\end{pmatrix}$$
	satisfies the conditions of the last two theorems, in spite of $\dim R(\bbf(v)) < 2$ for all $v \in \R$.
\end{remark}

\subsection{Outline of the paper} This manuscript is organized as follows. In Section 2, we will demonstrate Theorem \ref{1.firstthm}. Subsequently, in Section 3, we will show how to reduce Theorem \ref{1.secondthm} to Theorem \ref{1.firstthm}. In Section 4, we will concisely delineate the proof of both Theorems \ref{1.firstparthm} and \ref{1.secondparthm}, once they are almost identical the corresponding arguments of Theorems \ref{1.firstthm} and \ref{1.secondthm}. Finally, in Section \ref{1.secremarks}, we will discuss several details of the statement and proofs of such theorems; in particular, we will compare these results with theories of \textsc{P.-L. Lions--B. Perthame--E. Tadmor} \cite{LPT} and of \textsc{E. Tadmor--T. Tao} \cite{TT}.

\section{Proof of Theorem \ref{1.firstthm}}

First of all, passing to a subsequence if necessary, we may assume that, for all $j \in \JJ$, $(g_{j, n})_{n \in \N}$ and $(h_{j, n})_{n \in \N}$ are convergent in $L^r(\Om; L^{q_j}(\R_t\X\R_x^N\X\R_v))$, and that $(\Phi_n)_{n \in \N}$ is equally convergent in $L^2(\Om\X[0,\infty); HS(\mathscr H; L^2(\R_x^N\X\R_v)))$. Accordingly, the conclusions of Theorem \ref{1.firstthm} will be accomplished once we verify that, for any $\varphi \in (L^1\cap L^\infty)(\R_t\X\R_x^N)$, the averages $\varphi \int_\R f_n \eta \, dv$ define a convergent sequence in $L^r(\Om; L^s(\R_t\X\R_x^N))$.

	\subsection{The decomposition of the average}

	In this subsection, we compartmentalize  $\int_\R  f_n \eta\, dv$ into components whose a priori estimates may be extracted from different hypotheses made in statement of the  Theorem \ref{1.firstthm}. In this fashion, the desired conclusion is established via a proper passage to the limit. 
	
	Let us define the differences
	\begin{equation}
	\ff_{m,n}(t,x,v) = f_m(t,x,v) - f_n(t,x,v). \label{1.2.defg}
	\end{equation}
	Once \eqref{1.eqf} is linear, the elementary harmonic analysis asserts that each $\ff_{m,n}$ obeys
	\begin{align}
	\bigg(\frac{\del }{\del t} + \abf \cdot \nabla_x - \bbf : D_x^2 \bigg)\ff_{m,n}   &= \sum_{j \in \JJ} ( - \Delta_{t,x} + 1)^{1/2}  \left[1 \pm \left(\frac{\del^{\lf}}{\del v^{\lf}} (-\Delta_v)^{\zz/2}  \right) \right] \gf_{m,n}^{(j)} \nonumber\\ &+ \sum_{j \in \JJ}\Pi_j(v) \left( \Delta_M \right)  \left[1 \pm \left(\frac{\del^{\lf}}{\del v^{\lf}} (-\Delta_v)^{\zz/2}  \right) \right] \hh_{m,n}^{(j)} \nonumber \\ 
	& + ( - \Delta_{x} + 1)^{1/2} \left[1 \pm \left(\frac{\del^{\lf}}{\del v^{\lf}} (-\Delta_v)^{\zz/2}  \right) \right] \left( \Psi_{m,n}  \frac{dW}{dt} \right),  \label{1.2.eqgmn}
	\end{align}
	with the indices $\lf \in \Z $ and $0 \leq \zz <1$ being such that $\lf + \zz = \ell$, the sign $\pm$ being
	$$ \pm =
	\begin{dcases*}
	+, &\text{if } $\lf \equiv 0 \text{ mod }  4$,\\
	\text{arbitrary,} &\text{if } $\lf \equiv 1 \text{ mod } 4 \text{ or } 3 \text{ mod } 4$, \text{and} \\
	-, &\text{if } $\lf \equiv 2 \text{ mod }  4$,
	\end{dcases*}
	$$
	and, at last, each  $(\gf_{m,n}^{(j)})_{m, n \in \N}$, $(\hh_{m,n}^{(j)})_{m,n\in\N}$ and $(\Psi_{m,n})_{m,n\in\N}$ satisfying for all $j \in \JJ$
	\begin{align}
	&\lim\limits_{m, n \to \infty } \bbE \bigg( \int_{\R_t} \int_{\R_x^N} \int_{\R_v} |\gf_{m,n}^{(j)}(t, x, v)|^{q_j}\, dv dx dt \bigg)^{r/q_j} = 0, \label{1.2.limg}\\
	&\lim\limits_{m, n \to \infty } \bbE \bigg( \int_{\R_t} \int_{\R_x^N} \int_{\R_v} |\hh_{m,n}^{(j)}(t, x, v)|^{q_j}\, dv dx dt \bigg)^{r/q_j} = 0, \text{ and} \label{1.2.limh}\\
	&\lim\limits_{m, n \to \infty } \bbE \int_{0}^\infty \left\Vert  \Psi_{m,n}(t) \right\Vert_{HS(\mathscr H; L^2(\R_x^N\X\R_v))}^2\, dt = 0. \label{1.2.limPhi}
	\end{align}
	
	\subsubsection{The mollification of the weigh function $\eta$.}
	Let us now introduce a certain smooth approximation of $\eta$ which will allows us to handle the operator $\frac{\del^\lf}{\del v^\lf}(-\Delta_v)^{\zz/2}$ via integration by parts. This mollification, which we will symbolize by $\eta_{\d,\g}$---as it will depend on two parameters $\g$ and $\d$---, has a quite special support, whose role in our analysis can hardly be exaggerated.

	\begin{lemma} \label{1.2.1.molleta}
		Let $N \geq 1$ be an integer, $1 < p < \infty$, $\eta \in L^{p'}(\R)$ have compact support, and $\bbf : \R \to \mathscr{L}(\R^N)$ be continuous matrix function which has a dichotomous range and satisfies the nontransiency condition in $\supp \eta$. Let $\chi > 0$ be given. 
		
		For any $0 < \d$ and $\g < 1$, there exist functions $\nn_\g$ and $\eta_{\d,\g}$ in $L^{p'}(\R)$ for which the following assertions hold.
		\begin{enumerate}
			\item[(a)] Regarding $\nn_\g$:
			\begin{enumerate}
				\item[(a.i)] $\nn_\g$ in $L^\infty(\R)$ with $\Vert \nn_\g \Vert_{L^\infty(\R_v)} \leq \g^{-\chi}$;
				\item[(a.ii)] $\supp \nn_\g \subset \supp \eta$;
				\item[(a.iii)] $\Vert \nn_\g - \eta \Vert_{L^{p'}(\R)} \to 0$ as $\g \to 0_+$.
			\end{enumerate}
			\item[(b)] Regarding $\eta_{\d, \g}$:
			\begin{enumerate}
				\item [(b.i)]   $\eta_{\d, \g} \in \Cc_c^\infty(\R)$ and $\Vert \eta_{\d,\g} \Vert_{L^\infty(\R)} \leq \Vert \nn_\g \Vert_{L^\infty(\R)}$;
				\item [(b.ii)]  $\supp \eta_{\d,\g} \subset \supp \eta + (-\d,\d)$ and is the disjoint union of two compact sets $K_h = K_h^{(\d)}$ and $K_p = K_p^{(\d)}$, for which
				\begin{equation}
				\begin{dcases} \label{1.2.1.alternative}
				\bbf(v) \equiv 0 \text{ identically if $v \in K_h$, and} \\
				\bbf(v) \geq \cc_\d P_M \text{ whenever $v \in K_p$,}
				\end{dcases}
				\end{equation}
				where $\cc_\d > 0$ depends only on $\d$, and $M$ is the effective range of $\bbf$;
				\item [(b.iii)] for any $0 < \g < 1$ fixed, $\Vert n_{\d,\g} - \nn_\g \Vert_{L^{p'}(\R)} \to 0$ as $\d \to 0_+$.
			\end{enumerate}
		\end{enumerate}
	\end{lemma}
	\begin{proof}
	In order to verify (a), it suffices to consider the truncations
		$$\nn_\g(v) = \begin{dcases}
		-\g^{-\chi} &\text{if $\eta(v) < -\g^{-\chi}$,} \\
		\eta(v) &\text{if $|\eta(v)| \leq \g^{-\chi}$, and} \\
		\g^{-\chi} &\text{if $\eta(v) > \g^{-\chi}$.}
		\end{dcases}$$
	The construction of $\eta_{\d, \g}$ is fairly more intricate. For this purpose, consider $(\varrho_{\ve})_{\ve > 0}$ to be standard mollifiers in the real line.
	
	Were it not for the asserted decomposition of the support of $\eta_{\d, \g}$, evidently we could have chosen this function to be $(\varrho_\d \star \eta_{\g})$. Indeed, if the boundary of $\{ v \in \supp \eta; \bbf(v) = 0 \}$ is empty, define $\eta_{\d, \g}$ as such. Otherwise, so as to obtain this extra attribute, let us localize $(\varrho_\d \star \eta_{\g})$ by means of the next proposition \textit{à} Whitney of \textsc{A.P. Calderón--A. Zygmund} \cite{CZ}, whose proof may also be found in the classic book of \textsc{E.M. Stein} \cite{Stn}.
	
	\begin{proposition}[The existence of the ``regularized distance'']
		Let $d$ be a positive integer, and $F \subset \R^d$ be a nonempty closed subset. There exists a continuous function $\dd : \R^d \to \R$ such that
		\begin{enumerate}
			\item $c_1 \operatorname{dist}(x, F) \leq  \dd(x) \leq c_2 \operatorname{dist}(x, F)$ for all $x \in \R^d$, 
			\item $\dd \in \Cc^\infty(\R^d \setminus F)$, and, for all multi-indices $\af = (\af_1, \ldots, \af_d)$,
			$$\left|(D^\af \dd) (x) \right| \leq B_\af  \operatorname{dist}(x, F)^{1 - |\af|} \text{ for all $x \in \R^N \setminus F$},$$
			where $c_1$, $c_2$, and $B_\af$ are positive constants which do not depend on $F$.
		\end{enumerate}
	\end{proposition}
	We will employ this result as follows. Put $d = 1$, and let $F$ be the boundary of $\{ v \in \R; \bbf(v) = 0 \}$. Once $F$ is a closed set, there exists a function $\dd(v)$ with the properties listed above.
	
	Given any $\ve > 0$, define $H_\ve : \R \to \R$ to be the regular approximations of the Heaviside function
	$$H_\ve(z) = \int_{0}^z \varrho_\ve(w - 2 \ve)\, dw,$$
	and introduce $\xi_\ve(v) = H_\ve(\dd(v))$. It is clear that $0 \leq \xi_\ve(v) \leq 1$ everywhere, and that $\xi_\ve(v) \to 1_{\R \setminus F}(v)$ pointwisely as $\ve \to 0_+$. In addition, for $\supp \varrho_\ve \subset (-\ve, \ve)$,  $\xi_\ve(v)$ actually vanishes if $\dist(v, F)$ is sufficiently small, hence $\xi_\ve \in \Cc^\infty(\R)$. Finally, because $F \cap \supp \eta$ is of measure zero (here is where the nontransiency condition is necessary),
	\begin{align}
	\Vert \xi_\d (\varrho_\d \star \nn_\g) - \nn_\g \Vert_{L^{p'}(\R)} &\leq \Vert \xi_\d \nn_\g - \nn_\g \Vert_{L^{p'}(\R)} + \Vert (\varrho_\d \star \nn_\g) - \nn_\g \Vert_{L^{p'}(\R)} \nonumber \\
	&\to 0 \text{ as $\d \to 0_+$}. \label{1.2.1.molletabiii}
	\end{align}
	Let us therefore define $\eta_{\d,\g}(v) = \xi_\d (v) (\varrho_\d \star \nn_\g)(v).$ Once now statements (b.i) and (b.iii) are easily verified for such $\eta_{\d,\g}$, all that remains to finalize the proof of this lemma is property (b.ii).
	
	To this end, perceive at first that $\supp \eta_{\d,\g} \subset \supp \eta + (-\d, \d)$ is a basic result in the theory of convolution integrals. Per the properties of $\xi_\d$, the support of $\eta_{\d,\g}$ is formed by the disjoint union of two closed sets, each of which, in virtue of the dichotomous range hypothesis, lies entirely in the interior of $\{ v \in \R; \bbf(v) = 0 \}$ or of $\{ v \in \R;  R(\bbf(v)) = M \}$. In case of the second alternative, being $\bbf(v)$ symmetric, $P_M \bbf(v) P_M$ can be seen as a linear isomorphism in $M$. Thus, the lower bound in \eqref{1.2.1.alternative} is derived from a simple continuity argument. 
\end{proof}

\subsubsection{The decomposition in the Fourier space.}

Likewise, it is crucial that we introduce the next partitioning in the frequencies variables, which depends how degenerate is Equation \eqref{1.eqf} in that given region. So as to express such a division, let us define three Fourier symbols. Henceforth, $M \subset \R^N$ will denote the effective range of $\bbf(v)$. Furthermore, recall the definition of the symbol $\LL(i\tau, i\k, v) = i(\tau + \abf(v)\cdot \k) + \k \cdot \bbf(v)\k$ as given in \eqref{2.defL}. 

\begin{definition} The symbols $(R\LL)(i\tau, i\k, v)$, $\wt \LL(i\tau, i\k, v)$ and $(\wt{R\LL})(i\tau,$ $i\k, v)$ ($\tau \in \R$, $\k \in \R^N$, and $v \in \R$) are defined as follows.
	\begin{enumerate}
		\item By $(R\LL)(i\tau, i\k, v)$, it will be understood the so-called \textit{restricted symbol}:
		\begin{equation}
		(R\LL)(i\tau, i\k, v) = i\big(\tau + (P_{M^\perp}\abf)(v) \cdot \k\big). \label{1.defRL}
		\end{equation}
		\item By $\wt \LL(i\tau, i\k, v)$, it will be understood the so-called \textit{normalized symbol}: 
		\begin{align}
		\wt\LL(i\tau, i\k, v) &= \LL\bigg( \frac{i\tau}{\sqrt{\tau^2 + |\k|^2}}, \frac{i\k}{\sqrt{\tau^2 + |\k|^2}},v \bigg). \label{1.defwtL}
		\end{align}
		\item By $(\wt {R\LL})(i\tau, i\k, v)$, it will be understood the so-called \textit{restricted normalized symbol}:
		\begin{equation}
		(\wt{R\LL})(i\tau, i\k, v) = (R\LL)\bigg(\frac{i\tau}{\sqrt{\tau^2 + |P_{M^\perp}\k|^2}}, \frac{i(P_{M^\perp}\k)}{\sqrt{\tau^2 + |P_{M^\perp} \k|^2}},v \bigg). \label{1.defwtrL}
		\end{equation}
	\end{enumerate}
\end{definition}
Choose two functions $\lambda$ and $\psi \in \Cc_c^\infty(\C; \R)$ such that
\begin{enumerate}
	\item $\lambda(z) = 1$ for $|z| < \frac{1}{2}$,
	\item $0\leq \lambda(z) \leq 1$ for $\frac{1}{2} \leq |z| \leq 1$, 
	\item $\lambda(z) = 0$ for $|z| > 1$, and
	\item $\lambda(z) +  \psi(z) = 1$ everywhere.
\end{enumerate}
For any $0 < \d$ and $\g < 1$, which will be fixed for now---but will be let go to $0$ eventually---, let us then write
\begin{equation*}
\ff_{m,n}(t,x,v) = \sum_{\nu=1}^4\ff_{m,n}^{(\nu)}(t,x,v),
\end{equation*}
where, with $\FF_{t,x}$ denoting the Fourier transform in $(t,x)$,
\begin{equation} \label{1.2.2.decompg}
\begin{dcases} 
\ff_{m,n}^{(1)} = \FF_{t,x}^{-1} \bigg[ \lambda\bigg( \frac{\sqrt{\tau^2 + |\k|^2}}{\g} \bigg) (\FF_{t,x}  \ff_{m,n}) \bigg],  \\
\ff_{m,n}^{(2)} = \FF_{t,x}^{-1} \bigg[ \psi\bigg( \frac{\sqrt{\tau^2 + |\k|^2}}{\g} \bigg) \lambda \bigg( \frac{\wt\LL(i\tau, i\k, v)}{\d} \bigg) (\FF_{t,x}  \ff_{m,n}) \bigg], \\
\ff_{m,n}^{(3)} = \FF_{t,x}^{-1} \bigg[ \psi\bigg( \frac{\sqrt{\tau^2 + |\k|^2}}{\g} \bigg) \psi \bigg( \frac{\wt \LL(i\tau, i\k, v)}{\d} \bigg) 
\\\quad\quad\,\,\,\quad\quad\quad\quad\quad\quad\quad\quad\quad\quad \lambda \bigg( \frac{(\wt{R\LL})(i\tau, i\k, v)}{\d} \bigg) (\FF_{t,x}  \ff_{m,n}) \bigg], \text{ and} \\
\ff_{m,n}^{(4)} = \FF_{t,x}^{-1} \bigg[ \psi\bigg( \frac{\sqrt{\tau^2 + |\k|^2}}{\g} \bigg) \psi \bigg( \frac{\wt \LL(i\tau, i\k, v)}{\d} \bigg)\\\quad\quad\,\,\,\quad\quad\quad\quad\quad\quad\quad\quad\quad\quad \psi \bigg( \frac{(\wt{R\LL})(i\tau, i\k, v)}{\d} \bigg) (\FF_{t,x}  \ff_{m,n}) \bigg].
\end{dcases} 
\end{equation}
Even though neither $\wt\LL(i\tau, i\k, v)$ nor $(\wt{R\LL})(i\tau, i\k, v)$ are defined in the entire space $\R_\tau\X\R_\k^N\X\R_v$, this does not pose a problem, as their domain is of total measure nonetheless. Recall that it is admissible to take the spatio-temporal Fourier transform of $\ff_{m,n}$, as it almost surely lies in $L^p(\R_t\X\R_x^N\X\R_v)$ and, consequently, defines almost surely a tempered distribution. The tacit affirmation that each $\ff_{m,n}^{(\nu)}$ is indeed a function will be justified afterwards.

\subsubsection{Conclusion.}
All things considered, we thus establish the decomposition
\begin{align}
\int_\R \eta\ff_{m,n}\, dv &= \int_{\R}  \ff_{m,n}(\eta - \eta_{\d, \g}) \, dv + \int_{\R}  \ff_{m,n}^{(1)}\eta_{\d, \g} \, dv \nonumber\\ &\quad\quad\quad\quad\quad\quad  + \int_{\R}  \ff_{m,n}^{(3)}\eta_{\d, \g} \, dv +  \int_{\R}  \ff_{m,n}^{(4)}\eta_{\d, \g} \, dv \nonumber \\
&\stackrel{\text{def}}{=} \vf_{m,n}^{(0)} + \vf_{m,n}^{(1)} + \vf_{m,n}^{(2)} + \vf_{m,n}^{(3)} + \vf_{m,n}^{(4)}. \label{1.2.decompv0}
\end{align}
As a consequence, the definition of $\ff_{m,n}$ \eqref{1.2.defg} yields
\begin{equation}
\varphi \bigg(\int_{\R} f_m\eta \, dv - \int_{\R} f_n \eta\, dv \bigg) =  \sum_{\nu=0}^4\varphi\vf_{m,n}^{(\nu)},\label{1.1.decompv}
\end{equation}
in such a manner that our main objection is reduced to the extraction of a priori estimates in $L_\om^rL_{t,x}^s$ for each $\varphi \vf_{m,n}^{(\nu)}$ as $m$ and $n \to \infty$.

\subsection{The analysis of $\vf_{m,n}^{(0)}$.} 

\begin{proposition} \label{1.3.prop}
	There exists a constant $C = C\left(\Vert \varphi \Vert_{L_{t,x}^1 \cap L_{t,x}^\infty}, \sup_{\nu \in \N} \Vert f_\nu \Vert_{L_\omega^r L_{t,x,v}^p}\right)$  such that, for all $m$ and $n \in \N$,
	\begin{equation}
	\bbE \Vert \varphi \vf_{m,n}^{(0)} \Vert_{L^s(\R_t\X\R_x^N)}^r \leq C \Vert \eta_{\d, \g} - \eta \Vert_{L^{p'}(\R)}^r. \label{1.3.estvf0}
	\end{equation}
\end{proposition}

By virtue of Lemma \ref{1.2.1.molleta}, this is an adroit estimate as $\d$ and $\g$ separately tend to $0_+$. Before we demonstrate this bound for $\vf_{m,n}^{(0)}$, let us state the following elementary yet quite useful estimate, whose proof is an immediate corollary to the Hölder's inequality.

\begin{lemma} \label{1.3.lemma}
	For any exponent $1 \leq \sg \leq \infty$, $\phi \in L^{\sg'}(\R_v)$ and $\Lambda \in L^\sg(\R_t\X\R_x^N\X\supp \phi)$,
	\begin{equation*}
	\bigg\Vert \int_\R \phi(v) \Lambda(\,\cdot\,, \,\cdot\,,v)\, dv \bigg\Vert_{L^\sg(\R_t\X\R_x^N)} \leq \Vert \phi \Vert_{L^{\sg'}(\R_v)} \Vert \Lambda \Vert_{L^\sg(\R_t\X\R_x^N\X \supp \phi)}.
	\end{equation*}
	In particular, if $\Lambda \in L^\sg(\R_t\X\R_x^N\X\R_v)$,
	\begin{equation}
	\bigg\Vert \int_\R \phi(v) \Lambda(\,\cdot\,, \,\cdot\,,v)\, dv \bigg\Vert_{L^\sg(\R_t\X\R_x^N)} \leq \Vert \phi \Vert_{L^{\sg'}(\R_v)} \Vert \Lambda \Vert_{L^\sg(\R_t\X\R_x^N\X \R_v)}. \label{1.3.trivialineq}
	\end{equation}
\end{lemma}
\begin{proof}[Proof of Proposition \ref{1.3.prop}]
	Applying \eqref{1.3.trivialineq} to the definition of $\vf_{m,n}^{(0)}$, we deduce that
	\begin{align}
	\bbE \Vert \vf_{m,n}^{(0)} \Vert_{L^p(\R_t\X\R_x^N)}^r &\leq \Vert \eta_{\d, \g} - \eta \Vert_{L^{p'}(\R_v)}^r \bbE \Vert f_m - f_n \Vert_{L^p(\R_t\X\R_x^N\X\R_v)}^r \nonumber \\
	&\leq 2^r \Big(\sup_{\nu \in \N} \Vert f_\nu \Vert_{L_\omega^r L_{t,x,\xi}^p}^r \Big) \Vert \eta_{\d, \g} - \eta \Vert_{L^{p'}(\R_v)}^r; \nonumber
	\end{align}
	i.e., 
	\begin{equation*}
	\bbE \Vert \varphi \vf_{m,n}^{(0)} \Vert_{L^s(\R_t\X\R_x^N)}^r \leq 2^r\Vert \varphi \Vert_{L_{t,x}^1 \cap L_{t,x}^\infty}^r \Big(\sup_{\nu \in \N} \Vert f_\nu \Vert_{L_\omega^r L_{t,x,\xi}^p}^r \Big) \Vert \eta_{\d, \g} - \eta \Vert_{L^{p'}(\R)}^r,
	\end{equation*}
	which establishes \eqref{1.3.estvf0}.
\end{proof}

\subsection{The analysis of $\vf_{m,n}^{(1)}$.}
\begin{proposition} \label{1.4.prop0}
	Let $\phi \in \Cc_c^\infty(\C; \C)$, and $\ve > 0$. There exists a function $\mathfrak K \in \cap_{\nu = 0}^\infty W^{\nu,1}(\R_t\X\R_x^N)$ such that, for any $\Lambda \in \Ss(\R_t\X\R_x^N)$,
	$$\FF_{t, x}^{-1} \Bigg[\phi\bigg(\frac{\sqrt{\tau^2 + |\k|^2}}{\ve}\bigg) (\FF_{t, x} \Lambda ) \Bigg] = \ve^{N+1}(\mathfrak K(\ve\,\cdot\,,\ve\,\cdot\,) \star_{t,x} \Lambda).$$
	Moreover, for any integer $\nu \geq 0$,
	$$\Vert \mathfrak K \Vert_{W^{\nu,1}(\R_t\X\R_x^N)} \leq C(\nu, \supp \phi, \Vert \phi \Vert_{\Cc^{N+1}}).$$
\end{proposition}
\begin{proof}
	Put $\mathfrak G(\tau, \k) = \phi(\sqrt{\tau^2 + |\k|^2})$, and let $P(\tau, \k)$ be an arbitrary complex polynomial function. It is not hard to see that $P \mathfrak G \in W^{N+1,1}(\R_\tau \X \R_\k^N)$, and, for every multi-index $\af$ in $\R\X\R^N$ of length $N+1$, one has that
	$$\left| D^\af (P\mathfrak G)(\tau,\k) \right| \leq C_{A,P} \Vert \phi \Vert_{\Cc^{N+1}}  \frac{1_{(0,A)}(\sqrt{\tau^2 + |\k|^2})}{(\tau^2 + |\k|^2)^{\frac{N}{2}}},$$
	where $A > 0$ is any real number for which $\phi(z) = 0$ if $|z| > A$. Thus, $P \mathfrak G \in W^{N+1, \sg}(\R_\tau \X \R_\k^N)$ for any $1 \leq \sg < \frac{N+1}{N}$. As a result, the Haussdorf--Young inequality mingled with the Riemann--Lebesgue lemma asserts that $\mathfrak K = \FF_{t, x}^{-1} \mathfrak G$ satisfies the pointwise estimate
	$$|(D^\bb \mathfrak K)(t,x)| \leq \frac{H_\bb(t,x)}{(1 + \sqrt{t^2 + |x|^2})^{N+1}} \text{ for all $(t,x) \in \R_t\X\R_x^N$},$$
	where $\bb$ is any multi-index in $\R\X\R^N$, and $H_\bb \in L^{\mathfrak t}(\R_t\X\R_x^N)$ for $N+1 < \mathfrak t \leq \infty$ with $\Vert H_\bb \Vert_{L_{t,x}^{\mathfrak t}} \leq C(\bb, \mathfrak t, \Vert \phi \Vert_{\Cc^{N+1}}, \supp \phi)$. The desired conclusion now follows from the Hölder's inequality and the Fourier analysis operational rules.
\end{proof}

 \begin{remark}
	The argument above would have also been greatly simplified, had one assumed that $\phi$ is constant near the origin (as, for instance, $\lambda$ is); indeed, in this case $\mathfrak G \in \Cc_c^\infty(\R_\tau \X \R_\k^N)$, hence $\mathfrak K \in \Ss(\R_t\X\R_x^N)$. In spite of this, we have opted for this proof, seeing that this result will be summoned in the next subsection as well.
\end{remark}

\begin{proposition} \label{1.4.prop}
	There exist a constant $C = C(\Vert \varphi \Vert_{L_{t,x}^p}, \Vert \eta \Vert_{L_v^{p'}},$ $ \sup_{\nu \in \N} \Vert f_\nu \Vert_{L_\omega^r L_{t,x,v}^p})$ and an exponent $\qq > 0$, such that, for all $0<\g<1$, and $m$ and $n \in \N$,
	\begin{equation}
	\bbE \Vert \varphi \vf_{m,n}^{(1)} \Vert_{L^s(\R_t\X\R_x^N)}^r \leq C\g^\qq. \label{1.4.estvf1}
	\end{equation}
\end{proposition}
\begin{proof}
	According to Proposition \ref{1.4.prop0}, 
	\begin{align*}
	\vf_{m,n}^{(1)}(t,x) &= \g^{N+1}\bigg(\int_{\R_v}(\mathfrak K(\g\,\cdot\,,\g\,\cdot\,) \star_{t,x} \ff_{m,n})(\,\cdot\,,\,\cdot\,,v)\eta_{\d,\g}(v)\,dv\bigg)(t,x) \\&= \g^{N+1} \bigg(\mathfrak K(\g\,\cdot\,,\g\,\cdot\,) \star \int_{\R_v} \ff_{m,n}(\,\cdot\,,\,\cdot\,,v)\eta_{\d, \g}\,dv \bigg)(t,x),
	\end{align*}
	Thus, applying the Young's inequality for convolutions and the trivial estimate \eqref{1.3.trivialineq}, we see that, for almost any $\om \in \Omega$,
	\begin{align*}
	\Vert \vf_{m,n}^{(1)} \Vert_{\Cc_0(\R_t\X\R_x^N)} &\leq \g^{\frac{N+1}{p}} \Vert \mathfrak K \Vert_{L_{t,x}^{p'}} \bigg\Vert \int_{\R_v} \ff_{m,n}(\,\cdot\,,\,\cdot\,,v)\eta_{\d, \g}\,dv \bigg\Vert_{L_{t,x}^p} \\
	&\leq \g^{\frac{N+1}{p}} \Vert \mathfrak K \Vert_{L_{t,x}^{p'}} \Vert \eta_{\d,\g} \Vert_{L_v^{p'}} \Vert \ff_{m,n} \Vert_{L_{t,x,v}^p}
	\end{align*}
	(notice that the Sobolev inequality implies that $W^{N+1, 1}_{t,x} \subset L_{t,x}^1 \cap L_{t,x}^\infty$). The asserted bound with $\qq = r \frac{N+1}{p}$ now follows from a joint application of the Hölder's inequality and Lemma \ref{1.2.1.molleta}.
\end{proof}

\begin{remark} \label{1.4.remarkifty}
	Were $(f_n)_{n \in \N}$ also bounded in $L_\om^r L_{t,x,v}^{\varsigma}$ for some $1 \leq \varsigma < p$, the Young's inequality for convolutions could have been invoked to refine \eqref{1.4.estvf1} into
	$$\bbE \Vert \varphi \vf_{m,n}^{(1)} \Vert_{L^p(\R_t\X\R_x^N)}^r \leq C \g^{r(N+1)\big( \frac{1}{\varsigma} - \frac{1}{p} \big)} \Vert \eta_{\d,\g} \Vert_{L_{v}^{\varsigma'}}^r \bbE \Vert \ff_{m,n} \Vert_{L_{t,x,v}^\varsigma}^r.$$
	Thus, estimating $\Vert \eta_{\d, \g} \Vert_{L_v^{\varsigma'}}^r \leq C\Vert \eta_{\d, \g} \Vert_{L_v^\infty}^{r/\varsigma'} \leq  C \g^{-r\chi/\varsigma'}$, we see that, provided that $\chi = \chi(p, \varsigma)$ is chosen sufficiently small,
	$$\bbE \Vert \varphi \vf_{m,n}^{(1)} \Vert_{L^p(\R_t\X\R_x^N)}^r \leq C\g^\qq$$
	for all $m$ and $n \in \N$, and $\g > 0$, with $C = C(\Vert \varphi \Vert_{L_{t,x}^\infty}, \sup_{\nu \in \N} \Vert f_\nu \Vert_{L_\omega^r L_{t,x,v}^{\varsigma}})$, and $\qq = \qq(p, \varsigma) > 0$.
\end{remark}

\subsection{The analysis of $\vf_{m,n}^{(2)}$.} \label{1.SubSecTT}
Let us recall some results arising from the \textsc{E. Tadmor--T. Tao} theory \cite{TT}. 

\begin{definition} \label{1.5.def}
	\,
	\begin{enumerate}
		\item 	A Fourier multiplier $m(\tau, \k)$ on $\R_\tau\X\R_\k^N$ is said to satisfy the \textit{truncation property} if, for any $\phi \in \Cc_c^\infty(\C;\C)$, $\ve > 0$, and $1 < \sg < \infty$, the formula
		\begin{equation}
		\Lambda \in \Ss(\R_t\X\R_x^N) \mapsto \FF_{t, x}^{-1} \bigg[ \phi \bigg(\frac{m(\tau, \k)}{\ve}\bigg) (\FF_{t, x} \Lambda) \bigg] \label{1.5.defTrun}
		\end{equation}
		defines a bounded linear operator in $L^\sg(\R_t\X\R_x^N)$ whose norm may depend on $\sg$, and on the support and $\Cc^\nu$--norm of $\phi$ for some nonnegative integer $\nu$, but not on $\ve > 0$.
		
		\item Let $m(\tau, \k, v)$ be a Fourier multiplier on $\R_\tau\X\R_\k^N$ depending on a parameter $v \in \R_v$. $m(\tau, \k, v)$ is said to satisfy the \textit{truncation property uniformly in $v$} if, given any compact subset $K \subset \R_v$, the symbol $(\tau, \k) \mapsto m(\tau, \k, v)$ satisfies the truncation property, and, while the norm of the resulting operator in \eqref{1.5.defTrun} may still depend on $\sg$, and on the support and $\Cc^\nu$--norm of $\phi$ for some nonnegative integer $\nu$, both this norm and $\nu$ remain uniformly bounded in $\ve > 0$ as $v$ ranges over $K$.
	\end{enumerate}
\end{definition}

Let us also remember the following generalization of the Mihlin multiplier theorem due to \textsc{P.I. Lizorkin} \cite{L}, whose statement we adapt from \textsc{F. Zimmermann} \cite{Z}. Other demonstrations and further improvements may also be found in \textsc{R. Haller--H. Heck--A. Noll} \cite{HHN}, \textsc{P.C. Kunstmann--L. Weiss} \cite{KW}, and the references therein. (Recollect that, for any $w \in \R_y^d$, the differential operator $\frac{\del}{\del w}$ is defined as $w \cdot \nabla_y$).

\begin{theorem} \label{thm.lizorkin}
	Let $d$ be a positive integer, and $m \in L_\loc^1(\R^d)$. Assume that there exists an orthonormal basis $e_1, \ldots, e_d$ of $\R^d$ such that, for any multi-index $\af = (\af_1, \ldots, \af_d)$ observing $\af \leq 1 = (1, \ldots, 1)$, one has that $$ \frac{\del^{\af_1 + \cdots \af_d}m}{\del e_1^{\af_1} \cdots \del e_d^{\af_d}} \in L_\loc^1(\R^d),$$ and
	\begin{equation}
	\sum_{\af \leq 1}\operatornamewithlimits{ess\,sup}_{y \in \R^d} \left| (y \cdot e_1)^{\af_1} \cdots (y \cdot e_d)^{\af_d} \frac{\del^{\af_1 + \cdots \af_d}m}{\del e_1^{\af_1} \cdots \del e_d^{\af_d}} (y) \right| = B < \infty. \label{lizorkinhyp}
	\end{equation}
	Then, for any $1 < \sg < \infty$, $m$ is an $L^\sg(\R^d)$--multiplier, and there exists a constant $C = C_{\sg,d} > 0$ such that
	\begin{equation}
	\Big\Vert \FF_{y}^{-1}\big[ m(\,\,\cdot\,\,) (\FF_{y}f)(\,\,\cdot\,\,) \big] \Big\Vert_{L^\sg(\R_y^d)} \leq C B \Vert f \Vert_{L^\sg(\R_y^d)} \text{ for all $f \in \Ss(\R_y^d).$} \label{lizorkincon}¨
	\end{equation}
\end{theorem}

Let us now show that the symbols employed in the decomposition \eqref{1.2.2.decompg} indeed have the truncation property.

\begin{proposition} \label{1.5.proposition1}
	The following statements hold.
	\begin{enumerate}
		\item The symbol $(\tau, \k) \in \R_\tau\X\R_\k^N \mapsto \sqrt{\tau^2 + |\k|^2}$ satisfies the truncation property.
		
		\item The normalized symbol $\wt\LL(i\tau, i\k, v)$ observes the truncation property uniformly in $v$.
		
		\item Likewise, the normalized restricted symbol $(\wt{R\LL})(i\tau, i\k, v)$ fulfills the truncation property uniformly in $v$.
	\end{enumerate}
\end{proposition}
\begin{proof}
	
	First of all, statement (1) is an obvious conclusion flowing from Proposition \eqref{1.4.prop0}. On the other hand, the verification of the second assertion is trivialized after the constatation of the following two facts.
	
	\textit{Claim \#1}: The symbols $m_1$ and $m_2 :  \left(\R_\tau \X \R_\k^N \setminus \{ 0 \}\right)\X\R_v \to \R$ given by
	$$
	\begin{dcases}
	m_h(\tau, \k, v) = \frac{\tau}{\sqrt{\tau^2 + |\k|^2}} + \abf(v) \cdot \frac{\k}{\sqrt{\tau^2 + |\k|^2}}, \text{and} \\
	m_p(\tau, \k, v) = \frac{\k \cdot \bbf(v) \k}{\tau^2 + |\k|^2}
	\end{dcases}
	$$
	satisfy the truncation property uniformly in $v \in \R$. (Indeed, this follows directly from Theorem \ref{thm.lizorkin}. So as to facilitate such an inspection, notice that one may assume without loss of generality that
	$$\begin{dcases*}
	m_h(\tau, \k, v) = \sqrt{1+|\abf(v)|^2}\frac{\tau}{\sqrt{\tau^2 + |\k|^2}}, \text{ and} \\
	m_p(\tau, \k, v) = \frac{\lambda_1(v) \k_1^2 + \cdots + \lambda_N(v)\k_N^2}{\tau^2 + |\k|^2},
	\end{dcases*}
	$$
	where $0 \leq \lambda_1(v) \leq \cdots \leq \lambda_N(v) = \Vert \bbf(v) \Vert_{\mathscr{L}(\R^N)}$).
	
	\textit{Claim \#2}: If $m_1$ and $m_2 : \left(\R_\tau \X \R_\k^N \setminus \{ 0 \}\right) \X \R_v  \to \R$ are two real-valued multipliers satisfying the truncation property uniformly on $v$, then so does the complex-valued multiplier $m(\tau, \k, v) = m_1(\tau, \k, v) + i m_2(\tau, \k, v).$ (The proof of this statement utilizes Fourier series and may be found in \textsc{Tadmor--Tao} \cite{TT}).
	
	This couple of claims shows asseveration (2), leaving us only to inspect the statement (3). Comprehending $(\wt{R\LL})(i\tau, i\k, v)$ as a multiplier in $\R_\tau \X M^\perp$, the demonstration that this symbol possesses the truncation property uniformly in $v$ becomes---aside from minor technicalities---parallel to the analysis already described, and because of that we will omit it. The proof is now complete. 
\end{proof}

\begin{remark}
	Observe that the statement (1) could have been proven via Theorem \ref{thm.lizorkin} (or the Mihlin--Hörmander theorem). Nevertheless, the presented reasoning, besides being certainly more elementary, shows that the endpoints $\sg = 1$ and $\sg = \infty$ in Definition \ref{1.5.def} are valid for the particular symbol $(\tau,\k) \mapsto \sqrt{\tau^2 + |\k|^2}$. 
	
	What is more, let us point out that Claim \#1 answers positively a question posed in \textsc{Tadmor--Tao} \cite{TT}; see also \textsc{R.J. DiPerna}--\textsc{P.-L. Lions}--\textsc{Y. Meyer} \cite{DLM}.
\end{remark}

\begin{lemma} \label{1.5.lemma1}
	There exist constants $C = C_{p}$ and $\pp = \pp_{p} > 0$, both independent of $0 < \d$ and $\g < 1$, such that, almost surely, and for all $m$ and $n \in \N$,
	\begin{align}
	\Vert \vf_{m,n}^{(2)} \Vert_{L^p(\R_t\X\R_x^N)} \leq C \Vert &\eta_{\d,\g} \Vert_{L^\infty(\R)} \Vert \ff_{m,n} \Vert_{L^p(\R_t\X\R_x^N\X\R_v)}\nonumber\\& \bigg( \sup_{\tau^2 + |\k|^2 = 1} \meas\Big\{ v \in \supp \eta_{\d, \g};  |\LL(i\tau, i\k, v)| \leq \d \Big\} \bigg)^{\pp}.  \label{1.5.lemma1eq}
	\end{align}
	As a result, for all $m$ and $n \in \N$,
	\begin{align}
	\bbE\Vert \varphi \vf_{m,n}^{(2)} &\Vert_{L^s(\R_t\X\R_x^N)}^r \nonumber\\&\leq C \Vert \eta_{\d,\g} \Vert_{L^\infty(\R)}^{r} \bigg( \sup_{\tau^2 + |\k|^2 = 1} \meas\Big\{ v \in \supp \eta_{\d, \g};  |\LL(i\tau, i\k, v)| \leq \d \Big\}\bigg)^{r\mathfrak p}, \label{1.5.estvf2}
	\end{align}
	where $C = C\big( \Vert \varphi \Vert_{L_{t,x}^1 \cap L_{t,x}^\infty}, \sup_{\nu \in \N} \Vert f_\nu \Vert_{L_\omega^r L_{t,x,v}^p} \big)$ is independent of $0 < \d$ and $\g < 1$.
\end{lemma}
\begin{proof}
	The result will follow from the investigation of the norm of the linear transformation	
	\begin{align*}
	(T_{\d, \g} f)(t,x)=  \FF_{t,x}^{-1} \bigg[ \int_{\R_v} \eta_{\d, \g}(v)\lambda \bigg(\frac{\wt\LL(i\tau, i\k, v)}{\d} \bigg) \psi \bigg(\frac{\sqrt{\tau^2 + |\k|^2}}{\g} \bigg) (\FF_{t,x} f)\, dv \bigg](t, x).
	\end{align*}
	Observe that, according the previous proposition---once that $\psi(\sqrt{\tau^2 + |\k|^2}/\g) = 1 - \l(\sqrt{\tau^2 + |\k|^2}/\g)$---, the trivial estimate \eqref{1.3.trivialineq} asserts that $T_{\d,\g} : L_{t,x,v}^\sg \to L_{t,x}^\sg$ is continuous for any $1 < \sg < \infty$, and
	\begin{equation}
	\Vert T_{\d,\g} \Vert_{\mathscr L(L_{t,x,v}^\sg; L_{t,x}^\sg)} \leq C_\sg \Vert \eta_{\d, \g} \Vert_{L^\infty(\R)} \label{1.5.Lqq}
	\end{equation}
	for some $C_\sg$ which is independent of $0 < \d$ and $\g < 1$.
	
	Let us consider initially the case $p = 2$. In this scenario, we may sharpen the trivial estimate \eqref{1.3.trivialineq} by means of the Plancherel identity, in order to obtain
	\begin{align}
	\Vert T_{\d,\g} f\Vert_{L^2(\R_t\X\R_x^N)}^2
	&\leq \int_{\R_\tau} \int_{\R_\k^N} \bigg( \int_{\{ w \in \R; |\wt\LL(i\tau, i\k, w)| \leq \d \}} |\eta_{\d,\g}(w)|^2 \,dw \bigg) \nonumber\\&\quad\int_{\R_v}\bigg| \lambda \left(\frac{\wt\LL(i\tau, i\k, v)}{\d} \right) \psi \bigg(\frac{\sqrt{\tau^2 + |\k|^2}}{\g} \bigg) (\FF_{t,x} f)(\tau,\k,v) \bigg|^2 \, dv d\k d\tau \nonumber\\
	&\leq \Vert \eta_{\d,\g} \Vert_{L^\infty(\R)}^2 \Vert f \Vert_{L^2(\R_t\X\R_x^N\X\R_v)}^2  \nonumber\\&\quad\quad \bigg( \sup_{\tau^2 + |\k|^2 = 1} \meas\Big\{ v \in \supp \eta_{\d, \g};  |\LL(i\tau, i\k, v)| \leq \d \Big\} \bigg).
	\end{align}
	In other words,
	\begin{align}
	\Vert T_{\d,\g} &\Vert_{\mathscr L(L_{t,x,v}^2; L_{t,x}^2)} \nonumber \\ &\leq \Vert \eta_{\d, \g} \Vert_{L^\infty(\R)} \bigg( \sup_{\tau^2 + |\k|^2 = 1} \meas\Big\{ v \in \supp \eta_{\d, \g};  |\LL(i\tau, i\k, v)| \leq \d \Big\} \bigg)^{1/2}. \label{1.5.lemma1prelimeq}
	\end{align}
	
	This proves \eqref{1.5.lemma1eq} if $p = 2$. For a general exponent $1 < p < \infty$, one can interpolate \eqref{1.5.lemma1prelimeq} with \eqref{1.5.Lqq} via the Riesz--Thorin theorem with exponents, say, $\sg = \frac{1+p}{2}$ if $1 < p < 2$, and $\sg = 2p$ if $2 < p < \infty$.
\end{proof}

Before we close this subsection, let us state and prove following topological fact which guarantees the utility of the estimate \eqref{1.5.estvf2}.

\begin{lemma} \label{1.5.lemma2}
	It holds that
	\begin{equation}
	\sup_{\tau^2 + |\k|^2 = 1} \meas\Big\{ v \in \supp \eta_{\d, \g};  |\LL(i\tau, i\k, v)| \leq \d \Big\} \to 0 \text{ as $\d \to 0_+$}. \label{1.5.lemma2eq}
	\end{equation}
\end{lemma}
\begin{proof}
	Assume, by absurd, that the conclusion \eqref{1.5.lemma2eq} is false, and denote by $\Sbb^N$ the sphere in $\R\X\R^N$. Under such an assumption, there would exist some $\vartheta > 0$, $\d_n \to 0_+$, and $(\tau_n, \k_n) \in \Sbb^N$ such that
	\begin{equation}
	\meas\Big\{ v \in \supp \eta_{\d_n,\g};  |\LL(i\tau_n, i\k_n, v)| \leq \d_n \Big\} \geq \vartheta \text{ for all $n \in \N$}. \label{1.5.lemma2eq'}
	\end{equation}
	Passing to a subsequence if necessary, we may assume that $(\tau_n, \k_n) \to (\tau_\infty, \k_\infty) \in \Sbb^N$. In light of the uniform continuity of $\LL(i\,\cdot\,, i\,\cdot\,, \,\cdot\,)$  over compact sets of $\Sbb^N\X\R_v$, and of the assertion (b.i) in Lemma \ref{1.2.1.molleta}, \eqref{1.5.lemma2eq'} implies that
	\begin{equation}
	\meas\Big\{ v \in \supp \eta + (-\d_n, \d_n);  |\LL(i\tau_\infty, i\k_\infty, v)| \leq \d_n + \ve_n \Big\} \geq \vartheta  \label{1.5.lemma2eq''}
	\end{equation}
	for all $n \in \N$ and some $\ve_n \to 0_+$. Notwithstanding, amalgamating  the Lebesgue dominated convergence theorem  and the nondegeneracy condition \eqref{1.nondeg},
	\begin{equation*}
	\lim_{n \to \infty }\meas\Big\{ v \in \supp \eta + (-\d_n, \d_n);  |\LL(i\tau_\infty, i\k_\infty, v)| \leq \d_n+ \ve_n \Big\} = 0,
	\end{equation*}
	which is a blatant contradiction of \eqref{1.5.lemma2eq''}. Once the absurd hypothesis cannot hold, the desired limit \eqref{1.5.lemma2eq} is thus established.
\end{proof}

\subsection{The analysis of $\vf_{m,n}^{(3)}$.} Let us reinterpret the results of the previous subsection to the context of $\vf_{m,n}^{(3)}$.
\begin{lemma} \label{1.6.lemma}
	The following statements hold.
	\begin{enumerate}
		\item There exists an exponent $\rr = \rr_p > 0$ independent of $0 < \d$ and $\g < 1$, such that, for all $m$ and $n \in \N$,
		\begin{align}
		\bbE\Vert \varphi \vf_{m,n}^{(3)} &\Vert_{L^s(\R_t\X\R_x^N)}^r \leq C \Vert \eta_{\d,\g} \Vert_{L^\infty(\R)}^{r}\nonumber\\& \Bigg( \sup_{\substack{(\tau, \k) \in \R\X M^\perp\\\tau^2 + |\k|^2 = 1}} \meas\Big\{ v \in \supp \eta_{\d, \g};  |\LL(i\tau, i\k, v)| \leq \d \Big\}\Bigg)^{r\mathfrak \rr}, \label{1.6.estvf3}
		\end{align}
		where $C = C\left( \Vert \varphi \Vert_{L_{t,x}^s \cap L_{t,x}^\infty}, \sup_{n \in \N} \Vert f_n \Vert_{L_\omega^r L_{t,x,v}^p} \right)$.
		\item It holds that
		\begin{equation}
		\sup_{\substack{(\tau, \k) \in \R\X M^\perp\\\tau^2 + |\k|^2 = 1}} \meas\Big\{ v \in \supp \eta_{\d, \g};  \LL(i\tau, i\k, v)| \leq \d \Big\} \to 0 \text{ as $\d \to 0_+$}. \label{1.6.lemma2eq}
		\end{equation}
	\end{enumerate}
\end{lemma}
\begin{proof}
	Observe that $(R\LL)(i\tau, i\k, v)$ and $(\wt{R\LL})(i\tau, i\k, v)$ can be seen as, respectively, $\LL(i\tau, i\k, v)$ and $\wt \LL(i\tau, i\k, v)$ restricted to $(\tau, \k, v) \in \R \X (M^\perp \setminus \{ 0 \}) \X \R_v $---hence their name. For it was already shown in Proposition \ref{1.5.proposition1} that $(\wt{R\LL})(i\tau, i\k, v)$ satisfies the truncation property uniformly in $v$, the derivation of the first statement becomes now indistinguishable from the proof of Lemma \ref{1.5.lemma1}.
	
	Finally, choosing $\k \in M^\perp$ in the nondegeneracy condition \eqref{1.nondeg}, we deduce that
	\begin{align*}
	\meas\Big\{ v \in \supp \eta; \LL&(i\tau, i\k, v) = 0 \Big\} = 0 \text{ $\forall (\tau, \k) \in \R\X M^\perp$ such that $\tau^2 + |\k|^2 = 1$.}
	\end{align*}
	Therefore, reprising the argument behind Lemma \ref{1.5.lemma2}, \eqref{1.6.lemma2eq} follows.
\end{proof}

\subsection{The analysis of $\vf_{m,n}^{(4)}$.} 
\subsubsection{Initial manipulations.}  It is not difficult to see 
$$(\tau, \k, v) \mapsto \psi\bigg(\frac{\sqrt{\tau^2 + |\k|^2}}{\g}\bigg)\psi\bigg(\frac{\wt\LL(i\tau, i\k, v)}{\d}\bigg) \frac{1}{\LL(i\tau,i\k, v)} $$
is a well-defined function in $(\Cc_\loc^{k, \a} \cap L^\infty)(\R_t\X\R_x^N\X\R_v)$ provided we understand it to be $0$ where $\LL(i\tau,i\k,v) = 0$. Accordingly, if we apply the Fourier transform to \eqref{1.2.eqgmn} and recall the definition $\ff_{m,n}^{(4)}$ as expressed in \eqref{1.2.2.decompg}, we thus are able to justify the formula
\begin{align}
(\FF_{t,x} \ff_{m,n}^{(3)}) &=  \psi\bigg(\frac{\sqrt{\tau^2 + |\k|^2}}{\g}\bigg) \psi\bigg(\frac{\wt\LL(i\tau,i\k, v)}{\d}\bigg) \psi \bigg( \frac{(\wt{R\LL})(i\tau,i\k, v)}{\d}\bigg)  \nonumber\\&\quad\quad\Bigg[\sum_{j \in \JJ}\frac{ \big( \tau^2 + |\k|^2 + 1 \big)^{1/2}}{\LL(i\tau,i\k,v)} \bigg(1 \pm \frac{\del^\lf}{\del v^\lf}(-\Delta_v)^{\zz/2}\bigg)(\FF_{t,x} \gf_{m,n}^{(j)}) \nonumber\\
&\quad\quad\quad - \sum_{j \in  \JJ} \frac{ \Pi_j(v) |P_M \k |^2}{\LL(i\tau,i\k,v)} \bigg(1 \pm \frac{\del^\lf}{\del v^\lf}(-\Delta_v)^{\zz/2}\bigg)(\FF_{t,x} \hh_{m,n}^{(j)}) \nonumber\\
&\quad\quad\quad + \frac{\big(|\k|^2 + 1 \big)^{1/4}}{\LL(i\tau,i\k,v)}\bigg(1 \pm \frac{\del^\lf}{\del v^\lf}(-\Delta_v)^{\zz/2}\bigg)\FF_{t,x}\left(\Psi_{m,n} \frac{dW}{dt} \right) \Bigg]. \nonumber 
\end{align}
Additionally, taking advantage that $\psi(\sqrt{\tau^2 + |\k|^2}/\g)$ cancels near the origin, we may substitute the term $(\tau^2 + |\k|^2 + 1)^{1/2}$ with $\sqrt{\tau^2 + |\k|^2}$ by modifying $\gf_{m,n}$. Therefore, this alteration yields the subdivision
\begin{align}
\vf_{m,n}^{(3)} &= \sum_{j \in \JJ} (I)_{m,n}^{(j)} + \sum_{j \in \JJ} (II)_{m,n}^{(j)} + (III)_{m,n}, \label{1.7.decompvf4}
\end{align}
where these parcels are given by
\begin{align*} 
(I)_{m,n}^{(j)} &=  \FF_{t,x}^{-1} \bigg\{ \int_\R \psi\bigg(\frac{\sqrt{\tau^2 + |\k|^2}}{\g}\bigg)  (\FF_{t,x} \wt\gf_{m,n}^{(j)})  \bigg(1 \pm (-1)^\lf \frac{\del^\lf}{\del v^\lf}(-\Delta_v)^{\zz/2}\bigg) \nonumber\\ &\bigg[ \eta_{\d,\g}(v)  \psi \bigg(\frac{\wt\LL(i\tau, i\k, v)}{\d}\bigg) \psi \bigg( \frac{(\wt{R\LL})(i\tau,i\k, v)}{\d}\bigg) \frac{\sqrt{\tau^2 +|\k|^2}}{\LL(i\tau, i\k, v)}    \bigg] \, dv \bigg\},\\
(II)_{m,n}^{(j)} &=  \FF_{t,x}^{-1} \bigg\{ \int_\R \psi\bigg(\frac{\sqrt{\tau^2 + |\k|^2}}{\g}\bigg) (\FF_{t,x} \hh_{m,n}^{(j)})   \bigg(1 \pm (-1)^\lf \frac{\del^\lf}{\del v^\lf}(-\Delta_v)^{\zz/2}\bigg) \nonumber\\ & \bigg[ \eta_{\d,\g}(v)  \psi \bigg(\frac{\wt\LL(i\tau, i\k, v)}{\d}\bigg) \psi \bigg( \frac{(\wt{R\LL})(i\tau,i\k, v)}{\d}\bigg) \frac{-\Pi_j(v)|P_M \k|^2}{\LL(i\tau, i\k, v)}    \bigg] \, dv \bigg\}, \text{ and}\\
(III)_{m,n} &=  \FF_{t,x}^{-1} \bigg\{ \int_\R \psi\bigg(\frac{\sqrt{\tau^2 + |\k|^2}}{\g}\bigg)  \FF_{t, x} \bigg( \Psi_{m,n} \frac{dW}{dt}\bigg)   \bigg(1 \pm (-1)^\lf \frac{\del^\lf}{\del v^\lf}(-\Delta_v)^{\zz/2}\bigg) \nonumber\\ & \bigg[ \eta_{\d,\g}(v)  \psi \bigg(\frac{\wt\LL(i\tau, i\k, v)}{\d}\bigg) \psi \bigg( \frac{\sqrt{\tau^2 + |\k|^2}}{\d}\bigg) \frac{\left(|\k|^2 + 1\right)^{1/4}}{\LL(i\tau, i\k, v)}    \bigg] \, dv \bigg\},
\end{align*}
and, for any $j \in \JJ$, $(\wt \gf_{m,n}^{(j)})$ still satisfies
\begin{align}
&\lim\limits_{m, n \to \infty } \bbE \bigg( \int_{\R_t} \int_{\R_x^N} \int_{\R_v} |\wt\gf_{m,n}^{(j)}(t, x, v)|^{q_j}\, dv dx dt \bigg)^{r/{q_j}} = 0. \label{1.7.limhtil}
\end{align}
Even though each $\wt \gf_{m,n}^{(j)}$ depends on $\g > 0$, this will not be of substance for now. Let us inspect each term $(I)_{m,n}^{(j)}$, $(II)_{m,n}^{(j)}$ and $(III)_{m,n}$ separately.

\subsubsection{The analysis of $(I)_{m,n}^{(j)}$.}

\begin{lemma}
	There exists a constant, independent of $m$ and $n \in \N$, such that, for all $j \in J$ and almost surely, 
	\begin{equation}
	\Vert(I)_{m,n}^{(j)}\Vert_{L^{q_j}(\R_t\X\R_x^N)} \leq C \Vert \wt \gf_{m,n}^{(j)} \Vert_{L^{q_j}(\R_t\X\R_x^N\X\R_v)}. \label{1.7.1.estlemma1}
	\end{equation}
	Consequently,
	\begin{equation}
	\lim_{m,n\to \infty} \bbE\Big\Vert\varphi\,\sum_{j\in \JJ}(I)_{m,n}^{(j)}\Big\Vert_{L^s(\R_t\X\R_x^N)}^r = 0. \label{1.7.1.estlemma2}
	\end{equation}
\end{lemma}
\begin{proof} 
	\textit{Step \#1}: In order to fix ideas, let us assume firstly that $\zz = 0$, i.e., $\lf = \ell$, so that, if $\ell \geq 1$, $\LL(i\tau, i\k, v)$ is of class $\Cc_{\loc}^{\ell-1,1}$ with respect to the velocity variable $v$ (notice that  $\LL(i\tau, i\k, v)$ is polynomial in $\tau$ and $\k$, and hence infinitely differentiable in these arguments). The crux of our reasoning is based on the construction of $\eta_{\d,\g}$---more specifically on assertion (b.ii) of Lemma \ref{1.2.1.molleta}---; thus, let us engage the same notations of this proposition here as well. Since the integrand defining $(I)_{m,n}^{(j)}$ is supported for $v \in \supp \eta_{\d,\g}$, we may bifurcate our attention between the alternatives that $v \in K_h$ or $v \in K_p$. 
	
	\textit{Step \#1.1}: Let us first investigate the case $v \in K_h$, in which, because $\bbf(v) = 0$, Equation \eqref{1.2.eqgmn} has a hyperbolic character. Letting $(\tau',\k')$ be the normalized frequency as defined in \eqref{2.deftauk'}, it holds that
	\begin{align*}
	\begin{dcases}
	\wt \LL(i\tau, i\k, v) = \LL(i\tau', i\k', v), \text{and}\\
	\LL(i\tau, i\k, v) = \sqrt{\tau^2 + |\k|^2} \LL(i\tau', i\k', v).
	\end{dcases}
	\end{align*}
	Observe that the last relations above remains true if one substitutes $v$ with another $w \in \R$, provided that $|v - w|< \dist(K_h, K_p)$.
	
	As a result, putting $\wt\psi(z) = \frac{1}{z} \psi(z)$ (which is, by all means, a regular function), each integrand of $(I)_{m,n}^{(j)}$ is transformed into
	\begin{align}
	\FF_{t,x}^{-1} &\bigg\{ \frac{1}{\d}\psi\bigg(\frac{\sqrt{\tau^2 + |\k|^2}}{\g}\bigg)  (\FF_{t,x} \wt\gf_{m,n}^{(j)})  \bigg(1 \pm (-1)^\ell \frac{\del^\ell}{\del v^\ell} \bigg) \nonumber\\&\quad\quad \bigg[ \eta_{\d,\g}(v) \wt\psi \bigg(\frac{\LL(i\tau', i\k', v)}{\d}\bigg) \psi \bigg(\frac{(\wt{R\LL})(i\tau', i\k', v)}{\d}\bigg) \bigg] \bigg\} \nonumber\\
	&\quad\quad\quad\quad= \sum_{\nu=0}^\ell \eta_{\d,\g}^{(\nu)}(v) \FF_{t,x}^{-1} \bigg\{ \psi\bigg(\frac{\sqrt{\tau^2 + |\k|^2}}{\g}\bigg) m_\nu(\tau, \k, \xi) (\FF_{t,x} \wt\gf_{m,n}^{(j)})  \bigg\}, \label{1.7.1.estKh}
	\end{align}
	with each $m_\nu(\tau, \k, v)$ being given by
	\begin{equation*}
	\begin{dcases}
	m_0(\tau, \k, v) =  \frac{1}{\d}\bigg( 1 \pm (-1)^\ell \frac{\del^\ell}{\del v^\ell}\bigg) \bigg[ \wt\psi \bigg(\frac{\LL(i\tau', i\k', v)}{\d}\bigg) \psi \bigg(\frac{(\wt{R\LL})(i\tau, i\k, v)}{\d}\bigg) \bigg], \text{ and}\\
	m_\nu(\tau, \k, v) = \pm\frac{(-1)^\ell}{\d} \binom{\ell}{\nu}\bigg(\frac{\del^{\ell-\nu}}{\del v^{\ell - \nu}}\bigg) \bigg[ \wt\psi \bigg(\frac{\LL(i\tau', i\k', v)}{\d}\bigg) \psi \bigg(\frac{(\wt{R\LL})(i\tau, i\k, v)}{\d}\bigg) \bigg]
	\end{dcases}
	\end{equation*}
	for $\nu = 1, \ldots, \ell$. On the grounds Theorem \ref{thm.lizorkin}, all of these symbols $m_{\nu}$ are $L^{q_j}(\R_t\X\R_x^N)$--multipliers for every $j \in \JJ$ and their norms are bounded in $v \in K_h$; in other words, for $v \in K_h$, \eqref{1.7.1.estKh} implies
	\begin{align}
	\bigg\Vert \FF_{t,x}^{-1} &\bigg\{ \frac{1}{\d}\psi\bigg(\frac{\sqrt{\tau^2 + |\k|^2}}{\g}\bigg)  (\FF_{t,x} \wt\gf_{m,n})  \bigg(1 \pm (-1)^\ell \frac{\del^\ell}{\del v^\ell} \bigg) \nonumber\\&\quad\quad\quad\quad \bigg[ \eta_{\d,\g}(v) \wt\psi \bigg(\frac{\LL(i\tau', i\k', v)}{\d}\bigg) \psi \bigg(\frac{(\wt{R\LL})(i\tau', i\k', v)}{\d}\bigg) \bigg] \bigg\} \bigg\Vert_{L^{q_j}(\R_t\X\R_x^N)} \nonumber\\
	&\quad\quad\quad\quad\quad\quad\quad\quad\leq C_j\bigg(\sum_{\nu=0}^\ell |\eta_{\d,\g}^{(\nu)}(v) | \bigg) \Vert  \wt\gf_{m,n}^{(j)}(\,\cdot\,, \,\cdot\,, v) \Vert_{L^{q_j}(\R_t\X\R_x^N)} \text{ a.s.},\label{1.7.1.estKh'}
	\end{align}
	for all $j \in \JJ$, where $C_j$ does not depend on $v \in K_h$, and on $m$ and $n \in \N$.
	
	\textit{Step \#1.2}: The last estimate is enough to control the integral $(I)_{m,n}^{(j)}$ when $v$ ranges over $K_h$. Let us now investigate the other dichotomic option: let $v \in K_p$ be given. Even though now there is no simplification in the integrand of $(I)_{m,n}^{(j)}$, we may still perform the necessary differentiations, arriving at the formula
	\begin{align}
	(I)_{m,n}^{(j)} &= \FF_{t,x}^{-1} \bigg[\pm(-1)^\ell \int_\R \eta_{\d,\g}^{(\ell)}(v) \psi\bigg( \frac{\wt \LL(i \tau, i\k, v)}{\d} \bigg)\psi\bigg( \frac{(\wt {R\LL})(i \tau, i\k, v)}{\d} \bigg)  \nonumber\\&\quad\quad\quad\quad \psi\bigg( \frac{\sqrt{\tau^2 + |\k|^2}}{\g} \bigg) \frac{\sqrt{\tau^2 + |\k|^2}}{\LL(i\tau, i\k, v)} (\FF_{t, x} \wt\gf_{m,n}^{(j)}) \, dv \bigg] + \big[\text{ similar terms }\big]. \label{1.7.1.IKp}
	\end{align} 
	Although the omitted parcels could have been explicitly expressed via the Leibniz's and Faà di Bruno's rules, all portions can be handled analogously. Thus, we will concentrate ourselves  with the sole portion above.
	
	We are thus led to examine the Fourier operator
	\begin{align}
	f \mapsto \FF_{t,x}^{-1}\bigg[ \psi\bigg( \frac{\wt \LL(i \tau, i\k, v)}{\d} \bigg) \psi\bigg( \frac{\sqrt{\tau^2 + |\k|^2}}{\g} \bigg)\psi\bigg( &\frac{(\wt {R\LL})(i \tau, i\k, v)}{\d} \bigg)  \nonumber\\&\quad\quad\frac{\sqrt{\tau^2 + |\k|^2}}{\LL(i\tau, i\k, \xi)} (\FF_{t, x} f)  \bigg]. \label{1.7.1.fouriermult}
	\end{align}
	In order to verify that such an expression defines an $L_{t,x}^{q_j}$--multiplier, let us first establish a simple bound which will be stated as Lemma, since it will again be instrumental later on as well.
	
	\begin{lemma} \label{1.7.boundlemma}
		There exists a constant $C = C(\d,\g) > 0$, such that
		\begin{align}
		\bigg|\psi\bigg( \frac{\wt \LL(i \tau, i\k, v)}{\d} \bigg) \psi\bigg( \frac{\sqrt{\tau^2 + |\k|^2}}{\g} &\bigg)\psi\bigg( \frac{(\wt {R\LL})(i \tau, i\k, v)}{\d} \bigg)\frac{\sqrt{\tau^2 + |\k|^2} + |P_M\k|^2}{\LL(i\tau, i\k, v)}\bigg| \nonumber \\
		&\leq C \text{ for all $(\tau, \k,v) \in \left( \R\X\R^N \setminus\{ 0 \} \X M  \right)  \X K_p$}. \label{1.7.bound} 
		\end{align}
	\end{lemma}
	\begin{proof}
		Fix $v \in K_p$. We can suppose that $(\tau, \k)$ is such that $\psi((\wt {R\cL})(i\tau, i\k, v)/\d)\neq 0$, otherwise the entire expression is equal to zero, trivializing the estimate. In this case,
		\begin{align*}
		|\tau + \abf(v) \cdot \k| &= \left|(\tau + \abf(v) \cdot P_{M^\perp} \k) + (\abf(v) \cdot P_M \k)\right| \\
		&\geq \frac{\d}{2}\sqrt{\tau^2 + |P_{M^\perp} \k|^2} - \bigg(\sup_{v \in K_p} |\abf(v)| \bigg) |P_M \k| \\
		&\geq \frac{\d}{2}\sqrt{\tau^2 + |P_{M^\perp} \k|^2} - \frac{\cc_\d}{2} |P_{M}\k|^2 - A,
		\end{align*}
		where $A$ depends only on $\cc_\d$ and $K_p$, thus solely on $\d$ and $\g$ (recall we are employing the notations of Lemma \ref{1.2.1.molleta}). Hence, from the trivial inequality $\frac{1}{\sqrt 2}(|a| + |b|) \leq \sqrt{a^2 + b^2}$, and the fact that $\k \cdot \bbf(v) \k \geq \cc_\d |P_M\k|^2$ as $v \in K_p$,
		\begin{align*}
		|\LL(i\tau, i\k, v) | \geq \frac{1}{\sqrt 2} \left( \frac{\d}{2} \sqrt{\tau^2 + |P_{M^\perp} \k|^2} + \frac{\cc_\d}{2} |P_M \k|^2 - A \right),
		\end{align*}
		concluding the existence of constants $B$ and $R > 0$, depending only on $\d$ and $\g$, such that
		\begin{equation*}
		|\LL(i\tau, i\k, v) | \geq B \Big( \sqrt{\tau^2 + |\k|^2} + |P_M \k|^2 \Big)
		\end{equation*}
		if $\sqrt{\tau^2 + |\k|^2} \geq R$ and $\psi((\wt {R\LL})(i\tau, i\k, v)/\d)\neq 0$. This shows \eqref{1.7.bound} for $\sqrt{\tau^2 + |\k|^2}$ sufficiently large.

		On the other land, because $\{ \frac{\g}{2} \leq \sqrt{\tau^2 + |\k|^2} \leq R  \} \X K_p$ is compact, the continuity of
		$$(\tau, \k, v) \mapsto \psi\bigg( \frac{\wt \LL(i \tau, i\k, v)}{\d} \bigg)\frac{\sqrt{\tau^2 + |\k|^2} + |P_M\k|^2}{\LL(i\tau, i\k, v)}$$
		in this region proves \eqref{1.7.bound} for $\sqrt{\tau^2 + |\k|^2}$ of ``intermediate size''. Finally, for $\sqrt{\tau^2 + |\k|^2} < \frac{\g}{2}$, $\psi(\sqrt{\tau^2 + |\k|^2}/\g) = 0$, and, therefore, the desired bound is immediate in this region as well. The lemma is hereby demonstrated. 
	\end{proof}
	
	As a consequence, in order to apply Theorem \ref{thm.lizorkin}, we can argue just as in Proposition \ref{1.5.proposition1}: choose an orthonormal basis $e_0, e_1, \ldots, e_N$ in $\R\X\R^N$ such that $e_0 = (1,0)$, and, for $1\leq \nu \leq N$, $e_\nu = (0, \phi_\nu)$, with $\phi_\nu$ belonging either to $M$ or $M^\perp$. In these coordinates, it is not troublesome to verify the estimate \eqref{lizorkinhyp} uniformly for $v \in K_p$. Hence, according to Theorem \ref{thm.lizorkin} and the bound \eqref{lizorkincon}, \eqref{1.7.1.fouriermult} indeed defines an $L_{t,x}^{q_j}$--multiplier whose norm is bounded for $v \in K_p$.
	
	Therefore, reprising this reasoning, and agglutinating all parcels, the $L_{t,x}^{q_j}$--norm of the left-side of \eqref{1.7.1.IKp} can be estimated by
	\begin{align}
	\leq C_j\bigg(\sum_{\nu=0}^\ell |\eta_{\d,\g}^{(\nu)}(v) | \bigg) \Vert  \wt\gf_{m,n}^{(j)}(\,\cdot\,, \,\cdot\,, v) \Vert_{L^{q_j}(\R_t\X\R_x^N)} \text{ a.s.},\label{1.7.1.estKp}
	\end{align}
	where $C_j > 0$ is uniform for $v \in K_p$, and $m$ and $n \in \N$. 
	
	\textit{Step \#1: (Conclusion).} Once \eqref{1.7.1.estKp} is exactly the same estimate as \eqref{1.7.1.estKh'}, it is valid for all $v \in \supp \eta_{\d,\g} = K_h \cup K_p$. Consequently, integrating in $v$, invoking the trivial estimate \eqref{1.3.trivialineq}, and taking the expected value, we deduce \eqref{1.7.1.estlemma1}. Lastly, \eqref{1.7.1.estlemma2} is a direct byproduct of \eqref{1.7.limhtil}.
	
	\textit{Step \#2:} Assume now the fractional case $0 < \zz < 1$. Then, equation \eqref{1.7.1.IKp} reads
	\begin{align}
	(I)_{m,n} &= \FF_{t,x}^{-1} \bigg[\pm(-1)^\lf \int_\R (-\Delta_v)^{\zz/2} \bigg\{ \eta_{\d,\g}^{(\lf)}(v)  \psi\bigg(\frac{\wt \LL(i\tau,i\k, v)}{\d} \bigg)\psi\bigg( \frac{(\wt {R\LL})(i \tau, i\k, v)}{\d} \bigg)  \nonumber \\&\quad\quad\quad\quad\psi\bigg(\frac{\sqrt{\tau^2 + |\k|^2}}{\g}\bigg) \frac{\sqrt{\tau^2 + |\k|^2}}{\LL(i\tau, i\k, v)} \bigg\} (\FF_{t,x} \wt \gf_{m,n}^{(j)})  \, dv \bigg] + \big[\text{ similar terms }\big]. \label{1.6.I1'}
	\end{align} 
	Once more, let us exclusively focus on the leading term. 
	
	First of all, recall that the operator $(-\Delta_v)^{\zz/2}$ can be defined for all sufficiently smooth functions $\phi : \R \to \R$ as
	\begin{equation}
	\Big((-\Delta_v)^{\zz/2} \phi\Big)(v) = c_\zz \int_{\R} \frac{\phi(v)-\phi(w)}{|v-w|^{1 + \zz}} \, dw; \label{1.7.defLapl}
	\end{equation}
	(see, e.g., \textsc{P.R. Stinga} \cite{Sti}). While the numerical constant  $c_{\zz}$ is given by
	$$c_\zz = \frac{2^\zz}{\sqrt \pi} \frac{\Gamma\left(\frac{1+\zz}{2}\right)}{|\Gamma\left(-\frac{\zz}{2}\right)|},$$ its precise value will not be needed. In contrast to the first step, we observe that it is not possible to detach $\eta_{\d,\g}(v)$ from the other factors, and so the inequality \eqref{1.3.trivialineq} is no longer of applicability here. Moreover, due to the nonlocality of the fractional Laplacean, $v$ now varies through the entire real line rather than on the compact $\supp \eta_{\d,\g} = K_h \cup K_p$. As a result, we have no other choice than to show that $(-\Delta_v)^{\zz/2} m$, with $m(\tau, \k, v)$ given by
	\begin{align}
	m(\tau, \k, v) &=   \eta_{\d,\g}^{(\lf)}(v)  \psi\bigg(\frac{\wt \LL(i\tau,i\k, v)}{\d} \bigg)\psi\bigg( \frac{(\wt {R\LL})(i \tau, i\k, v)}{\d} \bigg) \nonumber \\&\quad\quad\quad\quad\quad\quad\quad\quad\quad\quad\quad\quad\quad \psi\bigg(\frac{\sqrt{\tau^2 + |\k|^2}}{\g}\bigg)  \frac{\sqrt{\tau^2 + |\k|^2}}{\LL(i\tau, i\k, v)}, \label{1.7.1.defm}
	\end{align}
	is an $L_{t,x}^{q_j}$--multiplier with a well-behaved norm as $|v| \to \infty$.
	
	\begin{proposition} \label{1.7.1.prop}
		For any $v \in \R_v$ and $j \in \JJ$, $(\tau, \k) \mapsto \big((-\Delta_v)^{\zz/2} m\big)(\tau, \k, v)$ is an $L^{q_j}(\R_t\X\R^N)$--multiplier. Moreover, if $T_v$ is the associated linear transformation
		$$(T_v \phi)(t,x) = \FF_{t, x}^{-1} \Big\{\big((-\Delta_v)^{\zz/2} m\big)(\,\cdot\,, \,\cdot\,, v) \FF_{t, x}\phi(\,\cdot\,, \,\cdot\,) \Big\}(t,x),$$
		then, there exists a constant $C_j > 0$ such that
		\begin{equation}
		\Vert T_v \Vert_{\mathscr{L}(L_{t,x}^{q_j})} \leq \frac{C_j}{(1 + |v|)^{1 + \zz}} \label{1.7.1.decayest}
		\end{equation}
		for all $v \in \R$.
	\end{proposition}
	\begin{proof}
		\textit{Step \#1}: Let us first show that the $\big((-\Delta)^{\zz/2} m\big)(\tau, \k, v)$ is an $L^q(\R_t\X\R^N)$--Fourier multiplier for any $v \in \R$ and $1 < q < \infty$. 
		
		Write
		\begin{align}
		\Big((-\Delta_v)^{\zz/2} m\Big)(\tau, \k, v) &= c_\zz \int_{|w|<\ve_0} \frac{m(\tau,\k,v)- m(\tau,\k, v + w)}{|w|^{1 + \zz}} \, dw \nonumber\\ &\quad+ c_\zz \int_{|w|>\ve_0} \frac{m(\tau,\k,v)-m(\tau,\k, v + w)}{|w|^{1 + \zz}} \, dv, \label{1.7.1.decompLapl}
		\end{align}
		where $\ve_0$ is the least number between $\dist(K_h, K_p)$ and, say, $1$. Evidently, as $m(\tau,\k, v)$ has compact support in $v$ and is an $L_{t,x}^{q}$--multiplier whose norm is globally bounded in $v \in \R_v$, the second integral above poses no difficulty: it represents an $L_{t,x}^{q}$--multiplier as well. 
		
		On the other hand, for any fixed $(\tau, \k) \in \R \X \left(\R^N \setminus M\right)$ and any multi-index $\af$ in $\R\X\R^N$, the function $v \in \R \mapsto (D^\af m)(\tau, \k, v)$ lies in the Hölder class $\Cc_c^\a(\R_v)$. Thus, once that $\a > \zz$, the singular integral in \eqref{1.7.1.decompLapl} not only converges absolutely for any $v \in \R$, but also may be freely differentiated in $(\tau, \k)$ under the integral sign.
		
		Dividing between the cases $v \in K_h$ (in which $m$ is homogeneous of degree $0$), and $v \in K_p$ (for which one may justifiably employ Lemma \ref{1.7.boundlemma}), one can apply Theorem \ref{thm.lizorkin} to once more show that $\big((-\Delta_v)^{\zz/2} m\big)(\tau, \k, v)$ is an $L_{t,x}^{q}$--multiplier.
		
		\textit{Step \#2}: Let us now confirm the decay estimate \eqref{1.7.1.decayest}. Evidently, as a corollary of the previous argument, $\Vert T_v \Vert_{\mathscr{L}(L_{t,x}^{q})}$ is bounded for if $v \in \R$ also remains bounded.
		
		That said, let $L > 0$ be any number for which $\supp \eta_{\d, \g} \subset (-L,L)$, so that
		\begin{align*}
		\big((-\Delta_v)^{\zz/2} m\big)(\tau, \k, v) &= -c_\zz \int_{-L}^{L} \frac{m(\tau,\k, w)}{|v - w|^{1 + \zz}} \, dw
		\end{align*}
		whenever $|v| > 2L$. From this formula, it is easily seen the existence of some constant $C_q > 0$ such that  $$\Vert T_v\Vert_{\mathscr{L}(L_{t,x}^{q})} \leq \frac{C_q }{|v|^{1+\zz}} \text{ for $|v| > 2L$.}$$ 
		
		The amalgamation of the statements of the former two paragraphs yields \eqref{1.7.1.decayest}, proving hereby the proposition.
	\end{proof}
	
	Consequently, in virtue of the last lemma and the Minkowski's and Hölder's inequalities,
	\begin{align*}
	\bigg\Vert \int_\R \FF_{t,x}^{-1} \bigg\{ \Big((-\Delta_v)^{\zz/2} m\Big)(\,\cdot\,,\,\cdot\,,v) &(\FF_{t, x} \wt\gf_{m,n}^{(j)})(\,\cdot\,, \,\cdot\,, v) \, dv \bigg\} \bigg\Vert_{L^{q_j}(\R_t\X\R_x^N)} \\ &\leq C_j \int_\R \frac{\Vert \wt\gf_{m,n}^{(j)}(\,\cdot\,, \,\cdot\,, v) \Vert_{L^{q_j}(\R_t\X\R_x^N)}}{(1+|v|)^{1+\zz}} \, dv   \\
	&\leq C_j\Vert \wt\gf_{m,n}^{(j)} \Vert_{L^{q_j}(\R_t\X\R_x^N\X\R_v)} \text{ a.s.},
	\end{align*}
	where $C_j$ is independent of $m$ and $n$. Returning to \eqref{1.6.I1'} and repeating this investigation to each and every element defining $(I)_{m,n}^{(j)}$, we once more conclude the estimate \eqref{1.7.1.estlemma1}, and hence \eqref{1.7.1.estlemma2} per \eqref{1.7.limhtil}. The lemma is proven.
\end{proof}

\subsubsection{The analysis of $(II)_{m,n}^{(j)}$.}
The investigation of $(II)_{m,n}^{(j)}$ is virtually identical to the one of $(I)_{m,n}^{(j)}$; thus the details will be omitted. In spite of that, let us only indicate that, once $\Pi_j(v) = 0$ whenever $v \in K_h$, one needs to investigate the alternative $v \in K_p$.

\begin{lemma}
	There exists a constant $C$, independent of $m$ and $n \in \N$, such that, almost surely and for all $j \in \JJ$,
	\begin{equation*}
	\Vert(II)_{m,n}^{(j)}\Vert_{L^{q_j}(\R_t\X\R_x^N)} \leq C \Vert  \hh_{m,n}^{(j)} \Vert_{L^{q_j}(\R_t\X\R_x^N\X\R_v)}. \label{1.7.2.estlemma1}
	\end{equation*}
	Consequently,
	\begin{equation}
	\lim_{m,n\to \infty} \bbE\Big\Vert\varphi\sum_{j \in \JJ}(II)_{m,n}^{(j)}\Big\Vert_{L^s(\R_t\X\R_x^N)}^r = 0. \label{1.7.2.estlemma2}
	\end{equation}
\end{lemma}

\subsubsection{The analysis of $(III)_{m,n}$.}

\begin{lemma} \label{1.7.3.lemma}
	There exists a constant $C$, independent of $m$ and $n \in \N$, such that
	\begin{equation}
	\bbE\Vert(III)\Vert_{L^2(\R_t\X\R_x^N)}^2 \leq C \bbE \int_0^\infty \Vert  \Psi_{m,n}(t) \Vert_{HS(\mathscr H, L^2(\R_x^N\X\R_v))}^2\, dt. \label{1.7.3.estlemma1}
	\end{equation}
	Consequently,
	\begin{equation}
	\lim_{m,n\to \infty} \bbE\Vert\varphi\,(III)\Vert_{L^s(\R_t\X\R_x^N)}^r = 0. \label{1.7.3.estlemma2}
	\end{equation}
\end{lemma}

Before presenting the proof of this lemma, let us explicate and explore the meaning of the expression $\Psi_{m,n} \frac{dW}{dt}$ and its the Fourier transform. As \eqref{1.eqftil} suggests, $\Psi_{m,n} \frac{dW}{dt}$ is defined as the linear functional
\begin{equation}
\phi \in \Ss(\R_t\X\R_x^N\X\R_v) \mapsto \int_0^\infty \int_{\R_x^N} \int_{\R_v} \phi(t,x,v) \Psi_{m,n}(t,x,v) \,dvdxdW(t).  \label{1.7.3.def}
\end{equation}

\begin{proposition}
	$\Psi_{m,n} \frac{dW}{dt}$, given by formula \eqref{1.7.3.def}, is almost surely a tempered distribution in $\R_t\X\R_x^N\X\R_x^N$. Furthermore, its spatio-temporal Fourier transform $\FF_{t, x}(\Psi_{m,n} \frac{dW}{dt})$ is, almost surely, formally given by
	\begin{align}
	\FF_{t,x} \bigg(\Psi_{m,n} \frac{dW}{dt}\bigg)(\tau, \k, v) = \frac{1}{\sqrt{2\pi}} \int_0^\infty e^{-it\tau} (\FF_x \Psi_{m,n}) (t, \k, v)\, dW(t); \label{1.7.3.ftransf}
	\end{align}
	that is, for any $\phi \in \Ss(\R_\tau\X\R_\k^N\X\R_v)$ and almost surely,
	\begin{align}
	\bigg\langle \FF_{t,x} &\left(\Psi_{m,n} \frac{dW}{dt}\right), \phi \bigg\rangle_{\Ss', \Ss} \nonumber \\ &= \frac{1}{\sqrt{2\pi}} \int_{\R_\tau}   \int_0^\infty \int_{\R_\k^N} \int_{\R_v} e^{-it\tau} (\FF_x \Psi_{m,n})(\tau,\k,v) \phi(\tau,\k,v)\, dvd\k dW(t)d\tau. \label{1.7.3.ftransf'}
	\end{align}
\end{proposition}
\begin{proof}
	Pick $\phi \in \Ss(\R_t\X\R_x^N\X\R_v)$. For the Burkholder inequality asserts that
	\begin{align}
	\bbE \sup_{t > 0}\bigg\Vert \int_0^t \Psi_{m,n}(t')&\,dW(t') \bigg\Vert_{L^2(\R_x^N\X\R_v)}^2 \nonumber \\&\leq C \bbE \int_0^\infty \big\Vert \Psi_{m,n}(t') \big\Vert_{HS(\mathscr H, L^2(\R_x^N\X\R_v))}^2\, dt' < \infty, \label{1.7.burkholder}
	\end{align}
	one may combine the stochastic Fubini theorem (see, e.g., \textsc{P. E. Protter} \cite{P}) with the usual formula $\phi(t,x,v) = - \int_t^\infty \frac{\del \phi}{\del t'}(t',x,v)\,dt'$ to translate the right-hand side of \eqref{1.7.3.def} into
	\begin{equation}
	- \int_0^\infty \bigg( \int_{\R_x^N} \int_{\R_v} \frac{\del \phi}{\del t}(t',x,v) \bigg[\int_0^t \Psi_{m,n}(t,x,v) \,d W(t)\bigg] \,dvdx\bigg)dt'. \label{1.7.3.def'}
	\end{equation}
	Thus, thanks to \eqref{1.7.burkholder} again, it is not difficult to argue from the formula \eqref{1.7.3.def'} that indeed $\Psi_{m,n}\frac{dW}{dt}$ defines almost surely a tempered distribution. (In other words, we have shown the ``intuitive'' relation $(\Psi_{m,n}\frac{dW}{dt}) = \frac{\del}{\del t} (\int_0^t \Psi_{m,n}\, dW)$).
	
	Let us now establish \eqref{1.7.3.ftransf}. Via the stochastic Fubini theorem once more, it may be shown that
	\begin{align*}
	\bigg\langle \Psi_{m,n} \frac{dW}{dt} , \FF_{t, x} \phi \bigg\rangle  &= \int_{0}^\infty \int_{\R_x^N} \int_{\R_v}  (\FF_{t, x}\phi) \Psi_{m,n}\, dvdxdW(t) \\
	&= \int_{0}^\infty \int_{\R_x^N} \int_{\R_v}   (\FF_{ t}\phi) ( \FF_x\Psi_{m,n} )\, dvd\k dW(t) \\
	&= \frac{1}{\sqrt {2\pi}}\int_{\R_\tau} \bigg( \int_{\R_x^N} \int_{\R_v} \bigg[\int_{0}^\infty e^{-i t \tau} \big(\FF_{x} \Psi_{m,n}\big) \, dW(t) \bigg]\phi \, dv dx \bigg)d\tau,
	\end{align*}
	hence \eqref{1.7.3.ftransf'}.
\end{proof}

\begin{proof}[Proof of Lemma \ref{1.7.3.lemma}]
	On the strength of the previous proposition, we deduce that
	\begin{align}
	(III)_{m,n} &=\FF_{t,x}^{-1} \bigg\{ \int_\R m(\tau, \k, v) \bigg( \int_0^\infty e^{-it\tau} (\FF_{x}\Psi_{m,n})(t,\k, v) \, dW(t) \bigg)  \, dv \bigg\},  \label{1.7.4.III'}
	\end{align}
	where $m : \R_\tau \X \R_\k^N \X \R_v \to \C$ is given by
	\begin{align}
	m(\tau, \k, v) = \psi\bigg(\frac{\sqrt{\tau^2 + |\k|^2}}{\g}\bigg) \bigg(&1 \pm (-1)^\lf \frac{\del^\lf}{\del v^\lf}(-\Delta_v)^{\zz/2}\bigg) \bigg[ \eta_{\d,\g}(v)  \psi \bigg(\frac{\wt\LL(i\tau, i\k, v)}{\d}\bigg) \nonumber\\& \psi \bigg(\frac{\wt{R\LL}(i\tau, i\k, v)}{\d}\bigg) \psi \bigg( \frac{\sqrt{\tau^2 + |\k|^2}}{\g}\bigg) \frac{\left(|\k|^2 + 1\right)^{1/4}}{\LL(i\tau, i\k, v)}    \bigg] \label{1.7.4.defm}
	\end{align}
	(a formal fashion to prove \eqref{1.7.4.III'} can be found in \textsc{B. Gess--M. Hofmanová} \cite{GH}). Notice that, mingling the bound \eqref{1.7.bound} of Lemma \ref{1.7.boundlemma} and reasoning of Proposition \ref{1.7.1.prop}, it is not difficult to corroborate the existence of a constant $C = C_{\d,\g} > 0$ such that
	\begin{align}
	|m(\tau, \k, v)| \leq  C \frac{ 1_{(\frac{\g}{2}, \infty)}(\sqrt{\tau^2 + |\k|^2}) }{(1+|v|)^{1 + \zz}} \frac{\sqrt{1 + |\k|}}{\sqrt{\tau^2 + |\k|^2}} \label{1.7.3.bound}
	\end{align}
	for all $(\tau, \k, v) \in \left(\R\X\R^N \setminus \{ 0 \}  \X M\right) \X\R$. Hence, a joint application of the Plancherel formula, the Cauchy--Schwarz inequality, \eqref{1.7.3.bound}, the Fubini theorem, and the Itô isometry to \eqref{1.7.4.III'} yields
	\begin{align}
	\bbE \Vert (III&)_{m,n} \Vert_{L^2(\R_t\X\R_x^N)}^2 \nonumber\\ &= \bbE \int_{\R_\tau} \int_{\R_\k^N}  \bigg|\int_{\R} m(\tau, \k, v) \bigg( \int_0^\infty e^{-it\tau} (\FF_{x}\Psi_{m,n})(t,\k, v) \, dW(t) \bigg)  \, dv  \bigg|^2\,d\k d\tau \nonumber \\
	&\leq \bbE \int_{\R_\tau} \int_{\R_\k^N}  \bigg(\int_{\R} |m(\tau, \k, w)|^2 \, dw \bigg) \nonumber\\ &\quad\quad\quad\quad\quad\quad\quad\quad \bigg( \int_{\R} \bigg|  \int_0^\infty e^{-it\tau} (\FF_{x}\Psi_{m,n})(t,\k, v) \, dW(t) \bigg|^2\ dv \bigg)  \,   d\k d\tau \nonumber\\
	&\leq C \int_0^\infty \int_{\sqrt{\tau^2 + |\k|^2}\geq \frac{\g}{2}} \int_{\R_v}  \frac{1 + |\k|}{\tau^2 + |\k|^2} \Vert  (\FF_{x}\Psi_{m,n})(t,\k,v)\Vert_{HS(\mathscr H; \R)}^2 \,dvd\k d\tau dt  \label{1.7.bound'},
	\end{align}
	where we introduced the notation $$\Vert (\FF_x\Psi_{m,n})(t,\k,v)\Vert_{HS(\mathscr H; \R)}^2 = \text{trace of } \Big\{(\FF_x\Psi_{m,n})(t, \k, v)^\star(\FF_x\Psi_{m,n})(t, \k, v) \Big\},$$
	which, by assumption, lies in $L_\om^1L_{t,\k,v}^1$. Integrating \eqref{1.7.bound'} firstly in the $\tau$-variable, we obtain that
	\begin{align*}
	\bbE \Vert (III)_{m,n} &\Vert_{L^2(\R_t\X\R_x^N)}^2 \\ &\leq C \int_0^\infty \int_{\R_\k^N} \int_{\R_v} (1 + |\k|) \zeta_\g(\k) \Vert  (\FF_{x}\Psi_{m,n})(t,\k,v)\Vert_{HS(\mathscr H; \R)}^2 \,dvd\k d\tau dt,   
	\end{align*}
	with the function $\zeta_\g : \R_\k^N \to \R$ being defined as
	$$ \zeta_\g(\k) \stackrel{\text{def}}{=} \int_{\sqrt{\tau^2 + |\k|^2} \geq \frac{\g}{2}} \frac{1}{\tau^2 + |\k|^2} \, d\tau = 
	\begin{cases}
	\frac{4}{\g} &\text{for } |\k| = 0, \\
	\frac{\pi - 2\arctan\sqrt{\frac{\g^2}{4|\k|^2} -1}}{|\k|} &\text{for } 0 < |\k| < \frac{\g}{2}, \text{ and} \\
	\frac{\pi}{|\k|}  &\text{for } |\k| \geq \frac{\g}{2}.
	\end{cases}$$
	Due to the boundedness of $(1 + |\k|)\zeta_\g(\k)$, \eqref{1.7.3.estlemma1} is thus verified. Finally, \eqref{1.7.3.estlemma2} follows from \eqref{1.2.limPhi}.
\end{proof}

	\subsubsection{The conclusion of the analysis of $\vf_{m,n}^{(4)}$.} Recalling the decomposition \eqref{1.7.decompvf4}, the limits \eqref{1.7.1.estlemma2}, \eqref{1.7.2.estlemma2}, and \eqref{1.7.3.estlemma2} affirm the next proposition.
\begin{lemma}
	It holds the limit
	\begin{equation}
	\lim_{m,n\to \infty} \bbE \Vert \varphi \vf_{m,n}^{(4)} \Vert_{L^s(\R_t\X\R_x^N)}^r = 0. \label{1.7.estvf4}
	\end{equation}
\end{lemma}

\subsection{The conclusion of the proof of Theorem \ref{1.firstthm}.} Returning to \eqref{1.1.decompv}, the merger of the estimates \eqref{1.3.estvf0}, \eqref{1.4.estvf1}, \eqref{1.5.estvf2}, \eqref{1.6.estvf3} and \eqref{1.7.estvf4} results in
\begin{align*}
\limsup_{m,n \to \infty} \bbE &\bigg\Vert \varphi \int_{\R_v} (f_m - f_n)\eta \, dv \bigg\Vert_{L^s(\R_t\X\R_x^N)}^r \leq C \Vert \eta_{\d, \g} - \eta \Vert_{L^{p'}(\R)}^r + C\g^\qq  \\
&+ C \Vert \eta_{\d,\g}\Vert_{L^\infty(\R)}^{r} \Bigg[\bigg( \sup_{\tau^2 + |\k|^2 = 1} \meas\Big\{ v \in \supp \eta_{\d, \g};  |\LL(i\tau, i\k, v)| \leq \d \Big\}\bigg)^{r\mathfrak p} \\
& \quad\quad\quad\quad\quad+\bigg( \sup_{\substack{(\tau, \k) \in \R\X M^\perp\\\tau^2 + |\k|^2 = 1}} \meas\Big\{ v \in \supp \eta_{\d, \g};  |\LL(i\tau, i\k, v)| \leq \d \Big\}\bigg)^{r\mathfrak \rr} \Bigg] ,
\end{align*}
where the positive constants $C$, $\pp$, $\qq$, and $\rr$ do not depend on the integers $m$ and $n$, nor on $0 < \d$ and $\g < 1$. Letting first $\d \to 0_+$, Lemmas \ref{1.2.1.molleta}, \ref{1.5.lemma2} and \ref{1.6.lemma} imply that
$$\limsup_{m,n \to \infty} \bbE \bigg\Vert \varphi \int_{\R_v} (f_m - f_n)\eta \, dv \bigg\Vert_{L^s(\R_t\X\R_x^N)}^r \leq C \Vert \nn_\g - \eta \Vert_{L^{p'}(\R)}^r + C\g^\qq.$$
Finally, passing $\g \to 0_+$ and applying Lemma \ref{1.2.1.molleta} one last time, we conclude
$$\lim_{m,n \to \infty} \bbE \bigg\Vert \varphi \int_{\R_v} (f_m - f_n)\eta \, dv \bigg\Vert_{L^s(\R_t\X\R_x^N)}^r = 0.$$
Therefore, the sequence of the averages $(\varphi \int_{\R_v} f_n \eta \, dv )$ is Cauchy on the Banach space $L^r(\Om; L^s(\R_t\X\R_x^N))$. Theorem \ref{1.firstthm} is hereby demonstrated. \qed

\section{Proof of Theorem \ref{1.secondthm}} \label{Sec.Proof2}
We will reduce  Theorem \ref{1.secondthm} to a corollary of Theorem \ref{1.firstthm}. Let $\theta \in \Cc_c^\infty(Q)$ be arbitrary, and consider also some $\vartheta \in \Cc_c^\infty(\R_v)$, such that
$$\begin{dcases}
0 \leq \vartheta \leq 1 \text{ everyhwhere, and} \\
\vartheta \equiv 1 \text{ in $\supp \eta + (-1,1)$}.
\end{dcases}$$
Put $\wt f_n(t,x,v) = \theta(t,x) \vartheta(v)^2 f_n(t,x,v)$. Hence, conserving the notation $\LL(i\tau, i\k, v) = i(\tau + \abf(v)\cdot \k) + \k \cdot \bbf(v)\k$, each $\wt f_n$ obeys the equation
\begin{align}
\frac{\del \wt f_n}{\del t} + \abf(v) \cdot \nabla_x \wt f_n - \bbf(v) : D_x^2 \wt f_n =  f_n&  \LL\bigg(\frac{\del}{\del t}, \nabla_x, v \bigg) (\theta\vartheta^2) + 2   \div_x  (f_n \bbf)  \cdot \nabla_x (\theta \vartheta^2)  \nonumber \\ &+ \sum_{j \in \JJ}\theta \vartheta^2 ( - \Delta_{t,x} + 1)^{1/2}	(- \Delta_v + 1)^{\ell/2}g_{j,n} \nonumber\\ &+ \sum_{j \in \JJ} \theta \vartheta^2 (\Pi_j(v) \Delta_M)	(- \Delta_v + 1)^{\ell/2} h_{j,n} \nonumber \\ &+ \theta \vartheta^2  ( - \Delta_{t,x} + 1)^{1/4} (- \Delta_v + 1)^{\ell/2} \Phi_n \frac{dW}{dt} \label{1.7.eqwtf}
\end{align}
almost surely in the sense of the distributions in $\R_t\X\R_x^N\X\R_v$.

\begin{lemma} \label{1.7.lemma}
	The equation \eqref{1.7.eqwtf} may be written as
	\begin{align}
	\frac{\del \wt f_n}{\del t} &+ \abf(v) \cdot \nabla_x \wt f_n - \bbf(v) : D_x^2 \wt f_n = \sum_{j \in \wt {\mathscr J}}( - \Delta_{t,x} + 1)^{1/2} (- \Delta_v + 1)^{\ell/2}  \wt g_{j,n} \nonumber \\ & -\sum_{j\in \wt {\mathscr J}}(\wt\Pi_j(v) \Delta_M)(- \Delta_v + 1)^{\ell/2} \wt h_{j,n} + ( - \Delta_{x} + 1)^{1/4} (- \Delta_v + 1)^{\ell/2} \wt\Phi_n \frac{dW}{dt}, \label{1.7.eqf}
	\end{align}
	where $\wt{ \mathscr J}$ is a finite index set such that, for any $j \in \wt{\mathscr J}$, 
	\begin{enumerate}
		\item $s \leq q_j < \infty$,
		\item  $(\wt g_{j,n})_{n \in \N}$ and $(\wt h_{j,n})_{n \in \N}$ are relatively compact sequences in $L^r(\Omega; L^{q_j}(\R_t\X\R_x^N\X\R_v))$,
		\item  $\wt {\Pi_j} \in \Cc_\loc^{k, \a}(\R)$ is such that $\supp \wt{\Pi_j} \subset \supp \bbf$, and
		\item  $(\wt {\Phi_n})_{n \in \N}$ is a predictable and relatively compact sequence in $L^2(\Omega\X [0,\infty)_t;$ $HS(\mathscr H;$ $L^2(\R_x^N \X \R_v)))$.
	\end{enumerate}
\end{lemma}

In order to rewrite each term in \eqref{1.7.eqwtf} to our liking, let us state and prove the next proposition.

\begin{proposition} \label{1.7.prop}
	Let $d$ be a positive integer, $\mathscr U \subset \R^d$ be a nonempty open set, $1 < p < \infty$ be an exponent, $\ell \geq 0$, and $(k, \a) \in \Z \X [0,1]$ satisfy the relation \eqref{1.condka}. 
	
	Then, for any $\Lambda$ belonging to the Sobolev space $W^{-\ell, p}(\mathscr U)$ and $\phi \in \Cc_c^{k, \a}(\R^d)$, the distribution $\phi \Lambda$ lies in $W^{-\ell,p}(\R^d)$. Moreover, there exists a constant $C = C(d)$ such that
	\begin{equation*}
	\Vert \phi \Lambda \Vert_{W^{-\ell,p}(\R^d)} \leq C  \Vert\phi\Vert_{\Cc^{k,\a}(\R^d)}  \Vert \Lambda \Vert_{W^{-\ell, p}(\mathscr U)}. \label{1.7.prop'}
	\end{equation*}
\end{proposition}
\begin{proof}
	On account of the definition of multiplication of distributions by regular functions and the duality relation $W^{-\ell, p}(\R^d) = W^{\ell,p'}(\R^d)^\star$, it suffices to show that there exists a constant $C = C(d)$ such that
	\begin{equation*}
	\Vert \phi u \Vert_{W^{\ell,p'}(\R^d)} \leq C \, \Vert\phi\Vert_{\Cc^{k,\a}(\R^d)} \, \Vert u \Vert_{W^{\ell, p'}(\mathscr U)},
	\end{equation*}
	for every $u \in W^{m,p}(\mathscr U)$. 
	
	If $m$ is an integer, then this inequality is derived directly from the Leibniz's rule. In the case that $m$ is not an integer, recall, since $\phi u \in W_0^{\ell,p'}(\mathscr U)$, its $W^{\ell,p'}$--norm is equivalent to
	$$\Vert \phi u \Vert_{L^{p'}(\R^d)} + \sum_{j=1}^d \left\Vert (-\Delta_y)^{\frac{\ell - \lfloor \ell \rfloor}{2}} \frac{\del^{\lfloor \ell \rfloor}}{\del y_j^{\lfloor \ell \rfloor}} (\phi u) \right\Vert_{L^{p'}(\R^d)}.$$ 
	Therefore, in virtue of the \textsc{L. Grafakos--S. Oh}'s Kato--Ponce inequality \cite{GOh}
	\begin{align*}
	\Vert (-\Delta)^{\sg/2} [FG] \Vert_{L^{p'}(\R^d)}  &\leq C \Vert (-\Delta)^{\sg/2} F \Vert_{L^{p'}(\R^d)} \Vert G \Vert_{L^\infty(\R^d)} \\& \quad + C \Vert (-\Delta)^{\sg/2} G \Vert_{L^{\infty}(\R^d)} \Vert F \Vert_{L^{p'}(\R^d)},
	\end{align*}
	which is valid for any $0 < \sg < 1$, and $F$ and $G \in \Ss(\R^d)$, the desired asseveration now follows.
\end{proof}

\begin{proof}[Proof of Lemma \ref{1.7.lemma}]
	Let us write \eqref{1.7.eqwtf} as 
	\begin{equation}
	\bigg(\frac{\del}{\del t} + \abf(v) \cdot\nabla_x - \bbf(v):D_x^2\bigg)\wt f_n = (I)_n + \sum_{j \in \JJ} (II)_{n}^{(j)} + \sum_{j \in \JJ} (III)_{n}^{(j)} + (IV)_n, \label{1.7.eqwtf'}
	\end{equation}
	where we are denoting
	\begin{equation} \label{1.7.decomp}
	\begin{dcases}
	(I)_n = f_n  \LL\bigg(\frac{\del}{\del t}, \nabla_x, v \bigg) (\theta\vartheta^2) + 2   \div_x  (f_n \bbf)  \cdot \nabla_x (\theta \vartheta^2)\\
	(II)_n^{(j)} = \theta \vartheta^2 ( - \Delta_{t,x} + 1)^{1/2}	(- \Delta_v + 1)^{\ell/2}g_{j,n} &[\,j\in \JJ\,] \\
	(III)_n^{(j)} =  \theta \vartheta^2 (\Pi_j(v) \Delta_M)	(- \Delta_v + 1)^{\ell/2} h_{j,n} &[\,j\in \JJ\,] \\
	(IV)_n = ( - \Delta_{x} + 1)^{1/4} (- \Delta_v + 1)^{\ell/2} \Phi_n \frac{dW}{dt}.
	\end{dcases}
	\end{equation}
	
	\textit{Step \#1}: First of all, let us inspect $(I)_n$. If one assumes hypothesis (a) in the statement of the theorem, it is clear that, applying Proposition \ref{1.7.prop} first to the $v$--variable and then to the $(t,x)$--ones,
	\begin{align}
	(I)_n &= (-\Delta_{t,x} + 1)^{1/2} (- \Delta_v + 1)^{\ell/2} Y_n(t,x,v) \label{1.7.i'},
	\end{align}
	where $(Y_n)_{n \in \N}$ is relatively compact in $L_\om^rL_{t,x,v}^p$. 
	
	On the other hand, if (b) holds, the very same argument applies to the first component, and
	\begin{equation}
	f_n  \LL\bigg(\frac{\del}{\del t}, \nabla_x, v \bigg) (\theta \vartheta^2) = (-\Delta_{t,x} + 1)^{1/2} (- \Delta_v + 1)^{\ell/2} Y_n'(t,x,v), \label{1.7.junk}
	\end{equation}
	with  $(Y_n')_{n \in \N}$ still being relatively compact in $L_\om^rL_{t,x,v}^p$. To facilitate the investigation of complementary parcel, notice that we can assume without loss of generality that $$M = \{ x = (x_1, \ldots, x_N) \in \R^N; x_\nu = 0 \text{ if $N' < \nu \leq N$}  \}$$ for some $ 0 \leq N' \leq N$. Hence,
	\begin{equation}
	2   \div_x  (f_n \bbf)  \cdot \nabla_x (\theta \vartheta^2) = -2 f_n\vartheta\bbf : D^2 (\vartheta \theta) +  2\div_x\big(  f_n\bbf \nabla_x(\theta\vartheta^2)\big). \label{1.7.junk'}
	\end{equation}
	Once again,
	\begin{equation}
	-2 f_n\vartheta\bbf : D^2 (\vartheta \theta) = (-\Delta_{t,x} + 1)^{1/2} (- \Delta_v + 1)^{\ell/2} Y_n''
	\end{equation}
	where $(Y_n'')_{n \in \N}$ is relatively compact in $L_\om^rL_{t,x,v}^p$. The second part, however, has the form
	$$2\div_x \big(  f_n \bbf \nabla_x (\theta \vartheta^2)  \big) = (\vartheta \bbf ) : P_M \nabla_x \otimes   	(-\Delta_{t,x} + 1)^{1/4}(-\Delta_v + 1)^{\ell/2}  K_n,$$
	where $(K_n)_{n\in\N} = \big(\, [K_n^{(1)}, \ldots, K_n^{(N)} ] \,\big)_{n\in \N}$ is relatively compact now in $L_\om^r(L_{t,x,v}^p)^N$. According to Theorem \ref{thm.lizorkin}, for any $1 \leq \nu \leq N'$,
	$$\bigg(\frac{\del}{\del x_\nu}\bigg) (-\Delta_{t,x} + 1)^{1/4} ( (-\Delta_{t,x} + 1)^{1/2} - \Delta_M)^{-1}$$
	defines a bounded operator in $L_{t,x}^p$. For this reason, writing $\bbf$ matricially as $\bbf = (\bbf_{\mu,\nu})_{1 \leq \mu,\nu \leq N}$,
	\begin{equation}
	2\div_x \big(  \bbf f_n  \nabla_x (\theta \vartheta^2)  \big) = \sum_{\mu, \nu = 1}^{N'} (\vartheta \bbf_{\mu,\nu}) ((-\Delta_{t,x} + 1)^{1/2} - \Delta_M) (-\Delta_v + 1)^{\ell/2}  \wt{K}_{\mu,\nu,n}, \label{1.7.junk''}
	\end{equation}
	with each $\wt{K}_{\mu,\nu,n}$ being relatively compact in $L_{t,x,v}^p$.
	
	Returning to the representation formulas \eqref{1.7.i'}---\eqref{1.7.junk''}, we conclude that
	$$(I)_n = \sum_{\mu, \nu = 1}^{N'} (\vartheta \bbf_{\mu, \nu})(-\Delta_M)(-\Delta_v + 1)^{\ell/2} K_{n, \mu, \nu}' + (-\Delta_{t,x} + 1)^{\ell/2}(-\Delta_v + 1)^{\ell/2} K_{n}'',$$
	where each and every $K_{n, \mu, \nu}'$ and $K_n''$ is relatively compact in $L_\om^rL_{t,x,v}^p$, as we wanted to show.
	
	\textit{Step \#2}: In an analogous fashion, all the other terms $(II)_n^{(j)}$, $(III)_n^{(j)}$, and $(IV)_n^{(j)}$ may be handled. Let us only point out a difference appearing in the analysis of $(III)_n^{(j)}$, in which we write 
	\begin{align*}
	(II)_n^{(j)} & = \theta \vartheta^2\Pi_j \Delta_M (- \Delta_v + 1)^{\ell/2} h_{j,n} \\
	& =  (\vartheta^2 \Pi_j ) \big( \Delta_M 	(- \Delta_v + 1)^{\ell/2} (\theta h_{j,n}) \big) \\&\quad - 2 (P_M \nabla_x) (\theta \Pi_j \vartheta^2) \cdot (P_M\nabla_x) (- \Delta_v + 1)^{\ell/2} h_{j,n} \\&\quad-  (\Delta_M)( \theta\Pi_j \vartheta^2 ) (- \Delta_v + 1)^{\ell/2} h_{j,n}(t,x,v).
	\end{align*}
	Evidently, the first term has the form $\vartheta \Pi_j (\Delta_M)(-\Delta_v + 1)^{\ell/2} H_{j,n} $, where $H_{j,n}$ is relatively compact in $L_\om^rL_{t,x,v}^{q_j}$. Moreover, according to Proposition \ref{1.7.prop}, the last two parts are equal to $ (-\Delta_{t,x} + 1)^{1/2}(-\Delta_v + 1)^{\ell/2} H_{j,n}' $, with  $H_{j,n}$ being again relatively compact in $L_\om^rL_{t,x,v}^{q_j}$. The lemma is hereby proven.
\end{proof}

Since trivially $(\wt f_n)$ is bounded in $L_\om^r L_{t,x,v}^p$ and $\theta \int_\R\eta f_n \, dv = \int_\R \eta \wt f_n\, dv $, the relative compactness of the averages now in $L_\om^rL_{t,x}^s$ is guaranteed by Theorem \ref{1.firstthm}. The final assertion in the statement of Theorem \ref{1.secondthm} is corollary of Proposition \ref{1.3.lemma}. \qed

\section{Proof of Theorems \ref{1.firstparthm} and \ref{1.secondparthm}}
We will only briefly depict the proof of Theorem \ref{1.firstparthm}, for the remaining details are indistinguishable from the ones found in Theorems \ref{1.firstthm} and \ref{1.secondthm}---as matter of fact, the verification of Theorem \ref{1.secondparthm} is sensibly more unproblematic than that of Theorem \ref{1.secondthm}.

First of all, we may assume that $M \neq \{ 0\}$, otherwise the conclusions can be derived from Theorems \ref{1.firstthm} and \ref{1.secondthm}. Furthermore, we may suppose, passing to a subsequence if necessary, to assume again that all $(g_{j,n})_{n \in \N}$ and $(\Phi_n)_{n\in \N}$ are convergent in their respective spaces. Whereas we will still define $\nn_\g$ as in Lemma \ref{1.2.1.molleta}, we will now simply put $\eta_{\d, \g} = (\varrho_\d \star\nn_\g)$, where $(\varrho_\ve)$ is a mollifier in the real line. Define also the Fourier multiplier
\begin{equation*}
(\wt{R\cE})(i\k, v) = \frac{ (P_M\k) \cdot \bbf(v) (P_M\k)}{(P_M\k\cdot P_M\k) } = \text{``the restricted normalized elliptic symbol''},
\end{equation*}
which can be shown to satisfy the truncation property (recall that $M^\perp \subset N(\bbf(v))$ for all $v \in \R$). Thus, if $\ff_{m,n} = f_m - f_n$, and $0 < \d$ and $\g < 1$ once more, introduce the Fourier decomposition
\begin{align*}
\ff_{m,n}^{(1)} &= \FF_{t, x}^{-1} \bigg[ \lambda\bigg( \frac{\sqrt{\tau^2 + |\k|^2}}{\g} \bigg) (\FF_{t, x}\ff_{m,n}) \bigg], \\
\ff_{m,n}^{(2)} &= \FF_{t, x}^{-1} \bigg[ \psi\bigg( \frac{\sqrt{\tau^2 + |\k|^2}}{\g} \bigg)\lambda\bigg( \frac{(\wt{R\cE})(i\k, v)}{\d} \bigg) (\FF_{t, x}\ff_{m,n}) \bigg], \\
\ff_{m,n}^{(3)} &= \FF_{t, x}^{-1} \bigg[ \psi\bigg( \frac{\sqrt{\tau^2 + |\k|^2}}{\g} \bigg) \psi\bigg(\frac{(\wt{R\cE})(i\k, v)}{\d} \bigg)\lambda\bigg( \frac{(\wt{R\LL})(i\tau, i\k, v)}{\d} \bigg) (\FF_{t, x}\ff_{m,n}) \bigg], \text{ and}\\
\ff_{m,n}^{(4)} &= \FF_{t, x}^{-1} \bigg[ \psi\bigg( \frac{\sqrt{\tau^2 + |\k|^2}}{\g} \bigg) \psi\bigg( \frac{(\wt{R\cE})(i\k, v)}{\d} \bigg) \psi \bigg( \frac{(\wt{R\LL})(i\tau, i\k, v)}{\d} \bigg)(\FF_{t, x}\ff_{m,n}) \bigg],
\end{align*}
where $\psi(z)$, $\lambda(z)$, $\wt{\LL}(i\tau, i\k, v)$ and $(\wt{R\LL})(i\tau,i\k,v)$ are also as before. Finally, write
\begin{align*}
\int_\R \ff_{m,n} \eta\, dv &= \int_\R \ff_{m,n}(\eta - \eta_{\d,\g})\, dv + \sum_{\nu=1}^4 \int_\R \ff_{m,n}^{(\nu)} \eta_{\d,\g}\,dv \stackrel{\text{def}}{=} \sum_{\nu=0}^4 \vf_{m,n}^{(\nu)}.
\end{align*}

Let $\varphi \in (L^1 \cap L^\infty)(\R_t\X\R_x^N)$ be given. Reprising the manipulations performed in the proof of Theorem \ref{1.firstthm}, for any $0 \leq \nu \leq 3$, $\varphi \vf_{m,n}^{(\nu)}$ have all an uniformly ``small'' $L_\om^rL_{t,x,v}^s$--norm as $\d$ and $\g$ tend to $0$ in a regulated manner. On the other hand, once the estimate \eqref{1.7.bound} now reads
\begin{align*}
\bigg| \psi\bigg(\frac{\sqrt{\tau^2 + |\k|^2}}{\g}  \bigg)\psi\bigg( \frac{(\wt {R\cE})(i\k, v)}{\d} \bigg)\psi \bigg( \frac{(\wt{R\LL})(i\k, v)}{\d} \bigg) \frac{\sqrt{\tau^2 + |\k|^2} + |P_M\k|^2}{\LL(i\tau, i\k, v)}\bigg| \leq C_{\d,\g} 
\end{align*}
for all $(\tau, \k, v) \in (\R\X\R^N \setminus( \R\X M^\perp \cup  \{ 0 \}\X M ) )\X(\supp \eta + (-1,1))$ (see \cite{N2} for an explicit calculation), it is clear that
$$\lim_{m,n \to \infty} \bbE \Vert \varphi \vf_{m,n}^{(4)} \Vert_{L^s(\R_t\X\R_x^N)}^r = 0,$$
in  spite of $\bbf(v)$ possibly not having total rank in $M$ and the right-hand side of \eqref{1.eqfpar} being relatively more singular. Based on these observations, Theorem \ref{1.firstthm} follows. \qed

\section{Last remarks} \label{1.secremarks}

\begin{remark}
	[On the spatially periodic case]
	It is not difficult to see that our results may be translated from $\R_x^N$ to $\bbT_x^N$, the $N$-dimensional torus, if one employs the so-called De Leeuw's theorem; see, e.g., \textsc{E. M. Stein--G. Weiss} \cite{SW}, theorem 3.8 in chapter VII. In this case, the local Theorems \ref{1.secondthm} and \ref{1.secondparthm} may even be strengthen as localizations in $x$ are no longer needed.
\end{remark}

\begin{remark} 
	[On the exponents $p$, $q_j$ and $r$] \label{Remark.LPT3}
	
	Should the stochastic terms $(\Phi_n)_{n \in \N}$ be absent in our averaging lemmas---i.e., we are in a deterministic setting---, not only the range  $1 \leq r < \infty$ is allowed, but one also can choose $s$ to be least number between $q_j$ and $p$. This represents a slight improvement on the exponent conditions of \textsc{P.-L. Lions--B. Perthame--E. Tadmor} \cite{LPT}, which assumed $\operatorname{card.} \JJ = 1$, $p = q_j$ and $1 < p \leq 2$. 
\end{remark}

\begin{remark}
	[On the exponents $p$, $q$ and $r$, part II]
	In a nutshell, the role of the function $\varphi \in L_{t,x}^{1}\cap L_{t,x}^\infty$ in Theorems \ref{1.firstthm} and \ref{1.firstparthm} was to convert all the $L^p$--, $L^{q_j}$-- and $L^2$--estimates into $L^s$--ones. Therefore, as Remark \ref{1.4.remarkifty} indicates, $\varphi$ is immaterial if such exponents are identical and one possesses an additional a priori estimate. 
	
	\begin{corollary}
		In the context of Theorems \ref{1.firstthm} and \ref{1.firstparthm}, assume in addition that
		\begin{enumerate}
			\item there exists some  $1 \leq \varsigma < p$ such that $(f_n)_{n \in \N}$ is also bounded in $L^r(\Om; L^{\varsigma}(\R_t$ $\X\R_x^N\X\R_v))$, and
			\item $p = q_j$ for all $j \in \JJ$.
		\end{enumerate} 
		Then, if either $p = 2$, or $\Phi_n \equiv 0$, the sequence of averages $(\int_{\R_v} f_n \eta\, dv)_{n \in \N}$ is relatively compact in $L^r(\Om; L^p(\R_t\X\R_x^N))$.
	\end{corollary}
	
	Although this assumption that $(f_n)_{n\in\N}$ is bounded in $L_\om^rL_{t,x,v}^{\varsigma}$ is commonly not found in the literature, in the applications to kinetic equations, the boundedness in $L_\om^rL_{t,x,v}^p$ is equivalent to one in $L_\om^rL_{t,x,v}^1$, wherefore it is not of extraordinary character. 
\end{remark}

\begin{remark}
	[On the exponents $p$, $q$ and $r$, part III]
	In the same spirit of the last two remarks, notice that essentially the low-frequency truncations $\lambda(\sqrt{\tau^2 + |\k|^2}/\g)$ are introduced so that one could to replace the operators $(-\Delta_{t,x} + 1)$ with its homogeneous counter-part $-\Delta_{t,x}$. Nevertheless, it is clear that, if $\bbf(v) \equiv 0$ and, in Equation \eqref{1.eqf}, $(-\Delta_{t,x} + 1)$ and $(-\Delta_{x} + 1)$ are substituted respectively by $-\Delta_{t,x}$ and $-\Delta_x$, then these truncations may be discarded.	One can thus deduce the next global averaging lemma, which recuperates a relative compactness result of \textsc{B. Perthame--P.E. Souganidis} \cite{PS}. 
	
	\begin{proposition} [The ``global'' hyperbolic averaging lemma]
		Given exponents $1 < p < \infty$, $1 \leq r \leq 2$ and $\ell \geq 0$, let $\abf \in \Cc_\loc^{k, \a}(\R; \R^N)$, where the real numbers $k$ and $\a$ satisfy the relation \eqref{1.condka}.
		
		Assume that, for any integer $n \in \N$, the equation
		\begin{align}
		\frac{\del f_n}{\del t} + \abf(v) \cdot& \nabla_x f_n = (-\Delta_{t,x})^{1/2} (- \Delta_v + 1)^{\ell/2}  g_n + ( - \Delta_{x})^{1/4} (- \Delta_v + 1)^{\ell/2} \Phi_n \frac{dW}{dt} \nonumber
		\end{align}
		is almost surely obeyed in $\Dd'(\R_t\X\R_x^N\X\R_v)$, where
		\begin{enumerate}
			\item  $(f_n)_{n\in\N}$ is a bounded sequence in $L^r\big(\Omega; L^p(\R_t\X\R_x^N\X\R_v))$,
			\item $(g_n)_{n \in \N}$ is a convergent sequence in  $L^r(\Omega; L^p(\R_t\X\R_x^N\X\R_v))$, and
			\item  $(\Phi_n)_{n \in \N}$ is a predictable and convergent sequence in $L^2(\Omega\X [0,\infty)_t;$ $HS(\mathscr H;$ $L^2(\R_x^N \X \R_v)))$.
		\end{enumerate}
		
		Finally, let $\eta \in L^{p'}(\R)$ have compact support, and presume that the \textnormal{nondegeneracy condition}
		\begin{align} 
		\meas \big\{ v \in \supp \eta; \tau + \abf(v) &\cdot \k = 0  \big\} = 0 \text{ for all $(\tau, \k) \in \R\X\R^N$ with $\tau^2 + |\k|^2 = 1$}\nonumber
		\end{align}
		holds.
		
		Then, if either $p = 2$, or $\Phi_n \equiv 0$, the sequence of averages $(\int_{\R_v} f_n \eta\, dv)_{n \in \N}$ is relatively compact in $L^r(\Om; L^p(\R_t\X\R_x^N))$.
	\end{proposition}
\end{remark}

\begin{remark}[On hypothesis (a) of Theorem \ref{1.secondthm}]
	Although the assumption (b) of Theorem \ref{1.secondthm} is more general than (a), the latter has its own appeal. First, notice that, if $\s(v)$ is the square-root of $\bbf(v)$, the condition on $\nabla_x f_n$ reads that
	$$(-\Delta_v + 1)^{-\ell/2} (-\Delta_{t,x} + 1)^{-1/2} \big(\div_x (\phi\s(v)) \cdot \,\operatorname{div}_x( f_n \z \s(v))\big)$$
	is relatively compact in $L^r(\Om; L^p(\R_t\X\R_x^N\X\R_v))$ for any $\phi \in \Cc_c^\infty(Q\X\R_v)$ and $\z \in \Cc_c^\infty(\R_v)$. This hypothesis is very akin to the impositions arising in the kinetic formulation of \textsc{G.-Q. Chen}--\textsc{B. Perthame} \cite{CP} (see also \textsc{A. Debussche}--\textsc{M. Hofmanová}--\textsc{J. Vovelle} \cite{DHV}); for this reason, it may be regarded as quite natural. 
	
	Additionally, as we have seen in the proof of Theorem \ref{1.secondthm}, should (a) hold, all the byproducts from the localization of $f_n$ itself can be rewritten as
	$$(-\Delta_{t,x} + 1)^{1/2}(-\Delta_v + 1)^{\ell/2} F_n,$$
	with $(F_n)_{n\in \N}$ being convergent in $L_\om^r L_{t,x,v}^p$---expressed in a different way, it does not bring forth derivatives of second order. 
\end{remark}

\begin{remark} [Equations with discontinuous coefficients]
	In certain models, one considers $\bbf(v)$ having the isotropic form \eqref{0.bsedimentation}, where $\mathbf q(v) = 0$ for $v$ belonging to some interval $I$, and $\mathbf q(v) = \mathbf q_c > 0$ for $v \notin I$, making thus \eqref{1.eqf} strongly degenerate; see, e.g., \textsc{R. Bürger--S. Evje--K. H. Karlsen} \cite{BEK} and \textsc{R. Bürger--K. H. Karlsen} \cite{BK}. Despite possessing now discontinuous coefficients, our theory may still apply to Equation \eqref{1.eqf} if one performs the following adjustment.
	
	Assume that, in any of the averaging lemmas we have studied here, all hypotheses are preserved, but one weakens the requirement on $\LL(i\tau, i\k, v)$ to $\abf \in (\Cc_\loc^{k,\a} \cap L_\loc^{p'}) (\R \setminus G; \R^N)$ and $\bbf \in (\Cc_\loc^{k,\a}\cap L_\loc^{p'})(\R \setminus G; \mathscr L(\R^N))$, where $G \subset \R$ is a closed set of zero Lebesgue measure. (The condition that $\abf(v)$ and $\bbf(v)$ belong to $L_\loc^{p'}$ is only made so as to Equations \eqref{1.eqf} and \eqref{1.eqfpar} to make sense).
	
	Following the proof of Lemma \ref{1.2.1.molleta}, one may construct a family of functions $(\Xi_{\ve})_{0 < \ve < 1}$, such that 
	\begin{enumerate}
		\item for all $0< \ve < 1$, $\Xi_{\ve} \in \Cc^\infty(\R_v)$,
		\item $0 \leq \Xi_{\ve}(v) \leq 1$ for all $0 < \ve < 1$ and $v \in \R$,
		\item for all $0 < \ve < 1$, there exists some $c_\ve > 0$ such that $\Xi_{\ve}(v) = 0$ if $\dist(v, G) < c_\ve$, and
		\item $\Xi_{\ve}(v) \rightarrow 1_{\R\setminus G}(v)$ for all $v \in \R$ as $\ve \to 0_+$.
	\end{enumerate}
	Repeating our techniques, it is not difficult to verify that $(\int_\R f_n \Xi_\ve \eta\,dv)$ is relatively compact in $L^r(\Om; L_\loc^s(Q))$ for any $0 < \ve < 1$ (here $Q$ may be $\R_t\X\R_x^N$). Therefore, in virtue of Proposition \ref{1.3.prop}, one derives that the original velocity averages $(\int_\R f_n \eta\,dv)_{n \in \N}$ are indeed relatively compact in $L^r(\Om; L_\loc^s(Q))$. 
	
	(Notice that in the preceding argument, it is not necessary to suppose that $\abf(v)$ and $\bbf(v)$ lie in, respectively, $L_\loc^\infty(\R_v; \R^N)$ and $L_\loc^\infty(\R_v; \mathscr{L}(\R^N))$. Generally, $f_n$ has uniformly bounded $L_\om^rL_{t,x,v}^p$--norms for all $1 \leq p \leq \infty$, permitting one to take $p' = 1$.)
\end{remark}

\begin{remark}[Comparison with the work of \textsc{P.L. Lions}, \textsc{B. Perthame} and \textsc{E. Tadmor}]
	Following the previous Remark \ref{Remark.LPT3}, let us continue juxtaposing our results with the classical averaging lemmas of \textsc{P.-L. Lions--B. Perthame--E. Tadmor} \cite{LPT}.
	
	Regarding the differences between our theory and theirs, let us mention this minor one: when $\ell$ was not an integer, they permitted the indices $(k,\a) = (\lfloor \ell \rfloor, \ell - \lfloor \ell \rfloor)$. Alas, this assumption could not be made in our arguments. Indeed, as we have seen, the operator $(-\Delta_v + 1)^{\ell/2}$ acts (via ``integrations by parts'') on the symbol $\LL(i\tau, i\k, v)$, forcing it to be Hölder--regular enough in order to $(-\Delta_v + 1)^{\ell/2}\LL(i\tau, i\k, v)$ to make sense. As a consequence, except when $\ell$ is an integer---which permits $(-\Delta_v + 1)^{\ell/2}$ to be transformed into a regular derivative---, $\abf(v)$ and $\bbf(v)$ need to have the sort of smoothness ``leeway'' we have imposed in \eqref{1.condka}; see, e.g., \textsc{P. R. Stinga} \cite{Sti}. To illustrate this point, notice that the function $\bbf(v) = |v|^{3/2}I_{\R^N}$ belongs to the Hölder class $C_\loc^{1, 1/2}(\R; \mathscr L(\R^N))$, but not to, say, $H_\loc^{3/2}(\R; \mathscr L(\R^N))$.
	
	In spite of this, we should point out that in most applications $\ell$ can be chosen as any number $>1$, hence the negligibility of this inconvenience.
	
	Therefore, having these observations in mind, we conclude that Theorem \ref{1.secondthm} may be understood as an extension of the hyperbolic compactness result of \textsc{Lions--Perthame--Tadmor} if $\bbf(v) \equiv 0$. The case $\bbf(v) \not\equiv 0$ is, however, distinct, for their theorem was stated for general diffusion matrices. Nevertheless, besides requiring $\bbf(v)$ to be smooth, they do not seem to allow a derivative of order higher than one in the forcing terms, which is instrumental for localization procedures---see the proof of Theorem \ref{1.secondthm}. 
	
	Curiously enough, there is one peculiar instance in which we can treat general diffusion matrices, even though this case is of no pertinence to the theory of entropy solutions.
	
	\begin{proposition} \label{1.8.propjunk}
		Let exponents $1 < p, q < \infty$, and $1 \leq r \leq 2$ be given. 
		Let also $\abf \in \Cc(\R; \R^N)$ and $\bbf \in \Cc(\R; \mathscr{L}(\R^N))$, with $\bbf(v)$ being nonnegative for all $v \in \R$.
		
		Assume that, for any $n \in \N$, the equation
		\begin{align*}
		\frac{\del f_n}{\del t} &+ \abf(v) \cdot \nabla_x f_n - \bbf(v) : D_x^2 f_n = ( - \Delta_{t,x} + 1)^{1/2} g_{n}  + ( - \Delta_{x} + 1)^{1/4} \Phi_n \frac{dW}{dt} 
		\end{align*}
		is almost surely obeyed in $\Dd'(\R_t\X\R_x^N\X\R_v)$, where
		\begin{enumerate}
			\item  $(f_n)_{n\in\N}$ is a bounded sequence in $L^r(\Omega; L^p(\R_t\X\R_x^N\X\R_v))$,
			\item  $(g_{n})_{n \in \N}$ is in $L^r(\Omega; L^{q}(\R_t$ $\X\,\,\R_x^N\X\R_v))$, and
			\item  $(\Phi_n)_{n \in \N}$ is a predictable and convergent sequence in $L^2(\Omega\X [0,\infty)_t;$ $HS(\mathscr H;$ $L^2(\R_x^N \X \R_v)))$.
		\end{enumerate}
		
		Finally, let $\eta \in L^{p'}(\R)$ have compact support, and presume that the \textnormal{nondegeneracy condition} \eqref{1.nondeg} holds.
		
		Then, with $s$ being the least number between $p$, $q$, and $2$, for any $\varphi \in  (L^1 \cap L^\infty)(\R_t\X\R_x^N)$, the sequence of averages $\left(\varphi \int_\R f_n\eta \, dv \right)_{n \in \N}$ converges in $L^r(\Om; L^s(\R_t\X\R_x^N))$. 
	\end{proposition}
	\begin{proof}[Sketch of the proof]
		Let us  keep the notations of the proof of Theorems \ref{1.firstparthm} and \ref{1.secondparthm}. If $\eta_{\d,\g}$ is the same as then, define now the decomposition
		\begin{equation*}
		\begin{dcases}
		\ff_{m,n}^{(1)} = \FF_{t, x}^{-1} \bigg[ \lambda\bigg( \frac{\sqrt{\tau^2 + |\k|^2}}{\g} \bigg) (\FF_{t, x}\ff_{m,n}) \bigg], \\
		\ff_{m,n}^{(2)} = \FF_{t, x}^{-1} \bigg[ \psi\bigg( \frac{\sqrt{\tau^2 + |\k|^2}}{\g} \bigg)\lambda\bigg( \frac{\LL(i\tau, i\k, v)}{\d\sqrt{\tau^2 + |\k|^2}} \bigg) (\FF_{t, x}\ff_{m,n}) \bigg], \text{ and} \\ 
		\ff_{m,n}^{(3)} = \FF_{t, x}^{-1} \bigg[ \psi\bigg( \frac{\sqrt{\tau^2 + |\k|^2}}{\g} \bigg)\psi\bigg( \frac{\LL(i\tau, i\k, v)}{\d\sqrt{\tau^2 + |\k|^2}} \bigg) (\FF_{t, x}\ff_{m,n}) \bigg],
		\end{dcases}
		\end{equation*}
		and write
		\begin{align*}
		\int_\R \ff_{m,n} \eta\, dv &= \int_\R \ff_{m,n}(\eta - \eta_{\d,\g})\, dv + \sum_{\nu=1}^3 \int_\R \ff_{m,n}^{(\nu)} \eta_{\d,\g}\,dv \stackrel{\text{def}}{=} \sum_{\nu=0}^3 \vf_{m,n}^{(\nu)}.
		\end{align*}
		The only term which needs some explanation is evidently $\ff_{m,n}^{(2)}$. Based on our techniques, it is not hard to see that $\LL(i\tau, i\k, v)/\sqrt{\tau^2 + |\k|^2}$ satisfies the truncation property uniformly on $v$. Moreover, it is not hard to see that
		\begin{align*}
		\bbE\Vert \vf_{m,n}^{(2)} \Vert_{L^p(\R_t\X\R_x^N)}^r \leq C \Vert &\eta_{\d,\g} \Vert_{L^\infty(\R)}^r \bbE \Vert \ff_{m,n} \Vert_{L^p(\R_t\X\R_x^N\X\R_v)}^r\nonumber\\& \bigg( \sup_{\tau^2 + |\k|^2 = 1} \meas\bigg\{ v \in \supp \eta_{\d, \g};  |\LL(i\tau, i\k, v)| \leq \frac{2\d}{\g}\bigg\} \bigg)^{r\pp}, 
		\end{align*}
		where $C$ and $\pp > 0$ are independent of $m$ and $n$. Since we pass $\d$ to zero prior to applying the same limit to $\g$, the factor $2\d/\g$ brings no hindrances.
		
		Furthermore, it is not hard to see that
		\begin{align*}
		\bigg|\psi\bigg( \frac{\LL(i \tau, i\k, v)}{\d\sqrt{\tau^2 + |\k|^2}} \bigg) &\psi\bigg( \frac{\sqrt{\tau^2 + |\k|^2}}{\g} \bigg)\frac{\sqrt{\tau^2 + |\k|^2}}{\LL(i\tau, i\k, v)}\bigg| \nonumber \\
		&\leq \frac{2}{\d} \text{ for all $(\tau, \k,v) \in \left( \R\X\R^N \setminus\{ 0 \}  \right)  \X \R_v$}.
		\end{align*}
		Hence, the proposition may be demonstrated following the same lines of the proof of Theorem \ref{1.firstthm}.
	\end{proof}
	
	The majority the remarks of this section also applies to Proposition \ref{1.8.propjunk}; for instance, a local theorem is available if one assumes that the same hypothesis (a) of Theorem \ref{1.secondthm} with $\ell = 0$. Notwithstanding, let us stress that the argument above is not valid for Equation \eqref{1.eqf} if $\ell > 0$, as we have discoursed in Subsection \ref{0.example}.
\end{remark}

\begin{remark}[In comparison with the work of \textsc{E. Tadmor} and \textsc{T. Tao}, and of \textsc{B. Gess} and \textsc{M. Hofmanová}] Even though the averaging lemmas of \textsc{E. Tadmor--T. Tao} \cite{TT} and \textsc{B. Gess--M. Hofmanová} \cite{GH} deal with the Sobolev regularity of the averages and hence are of different kind than ours, in several situations these type of result is used in the same context: to corroborate the existence of kinetic solutions to nonlinear degenerate convection--diffusion equations. It is thus interesting to contrast our theory with theirs.
	
	Well-understood, the crux of our argument is the regularizing effects of the Fourier quotient  $\frac{1}{\LL(i\tau, i\k, v)}$. In contrast, as we have commented in the Introduction, theirs was founded on dyadic decompositions and some uniform rates on the quantities $\om(J; \d)$ expressed in \eqref{0.2}.
	Hence, their method treats both the degree of $\LL(i\tau, i\k, v)$ and its behavior (parabolic or hyperbolic) quite indirectly and more abstractly. Even though this leads to a theorem enunciated in more broad terms, not only are their conditions much more arduous to be verified, but also all concrete examples provided by both works are also valid in our setting. 
	
	A particular and fascinating attribute of work of \textsc{B. Gess--M. Hofmanová} \cite{GH} is that, under some conditions on $\abf(v)$ and $\bbf(v)$, they could let the weight function $\eta$ not possess compact support, which seems to be a quite unprecedented assumption in the theory of the velocity averaging lemmas. Furthermore, they did not assume any Hölder regularity on $\abf(v)$ and $\bbf(v)$ (nevertheless, one usually employs some Hölder regularity in order to investigate \eqref{0.2}).
	
	Anyhow, it remains an intriguing conjecture to verify if the nontransient condition is somehow implicit in their hypotheses, or, conversely, if it is essential at all.
	
	A more tangible fashion to pose this conjecture is as follows. Like in Subsection \ref{0.example}, put $N =1$, let $\bbf \in \Cc_c^\infty(0,1)$ be a nonnegative function vanishing exactly in a Cantor set of positive measure in $[0,1]$, and define $\LL : \R_\tau \X \R_\k \X \R_v \to \C$ by
	$\LL(i\tau, i\k, v) = i(\tau + v \k) + \bbf(v) \k^2.$
	Evidently, $\bbf(v)$ does not obey the nontransient condition in $[0,1]$, consequently our theorem does apply to this particular symbol. Do, however, the hypotheses of \textsc{Tadmor--Tao} or \textsc{Gess--Hofmanová} apply? (Notice, since $\bbf(v)$ vanishes at infinite order in this Cantor set, it is not clear how to reproduce the analysis featured in section 4.2 of \cite{TT}; neither seems their condition (2.20) easily verifiable). If not, can an averaging lemma like Theorem \ref{1.firstthm} still be proven to this symbol?
\end{remark}
\textbf{Acknowledgments.} The author has been supported by CNPq Grant 140600/2017-5.

\end{document}